\documentclass[11pt]{amsart}

\def\N{{\mathbb N}}

\def\be#1{\begin{equation}\label{#1}}
\def\N{{\mathbb N}}
\def\R{{\mathbb R}}

\def\bc{\begin{center}}
\def\ec{\end{center}}

\numberwithin{equation}{section}
\newtheorem{definition}{Definition}[section]

\newtheorem{theorem}{Theorem}

\newtheorem{cor}[definition]{Corollary}
\newtheorem{lemma}[definition]{Lemma}
\newtheorem{proposition}[definition]{Proposition}
\newtheorem{example}[definition]{Example}
\newtheorem{remark}[definition]{Remark}

\newcommand{\sm}{\smallskip}

\newcommand{\ee}{\end{equation}}

\newcommand{\eea}{\end{eqnarray}}
\newcommand{\bean}{\begin{eqnarray*}}
\newcommand{\eean}{\end{eqnarray*}}

\newif\ifpctex

\newcommand{\bew}[1]{\begin{equation*}\label{#1}}
\newcommand{\bea}[1]{\begin{eqnarray}\label{#1}}

\newcommand{\beL}[2]{\begin{lemma}[#2]\label{#1}}
\newcommand{\beD}[2]{\begin{definition}[#2]\label{#1}}
\newcommand{\beT}[2]{\begin{theorem}[#2]\label{#1}}
\newcommand{\beP}[2]{\begin{proposition}[#2]\label{#1}}
\newcommand{\beC}[2]{\begin{cor}[#2]\label{#1}}

    \newcommand{\tvarepsilono}{{_{\displaystyle\longrightarrow\atop \varepsilon\to 0}}}

  \newcommand{\tno}{{_{\displaystyle\longrightarrow\atop n\to\infty}}}
  \newcommand{\tko}{{_{\displaystyle\longrightarrow\atop k\to\infty}}}

\usepackage{ulem}
\usepackage{color}
\definecolor{nb}{rgb}{.6,.176,1}
\definecolor{sienna}{rgb}{.92,.222,.176}
\definecolor{darkgreen}{rgb}{0,.5,0}

\def\em{\it}

\begin{document}
\author{Siva Athreya}
\address{Siva Athreya \\ Indian Statistical Institute
8th Mile Mysore Road  \\ Bangalore 560059, India.}
\email{athreya@isibang.ac.in}
\thanks{}

\author{Michael Eckhoff}

\author{Anita Winter}
\address{Anita Winter\\ Fakult\"at f\"ur Mathematik\\ Universit\"at Duisburg-Essen\\
Universit\"atsstrasse 2\\ 45141 Essen, Germany}
\thanks{}
\email{anita.winter@uni-due.de}


\title{Brownian Motion on $\R$-trees}

\date{\today}

\begin{abstract}
The real trees form a class of metric spaces that extends
the class of trees with edge lengths by allowing behavior such as infinite
total edge length and vertices with infinite branching degree.
We use Dirichlet form methods to construct Brownian motion on any given locally compact
$\R$-tree {$(T,r)$} equipped with a Radon measure $\nu$ {on $(T,{\mathcal B}(T))$}.
We specify a criterion under which the Brownian motion is recurrent or transient.
For compact recurrent $\R$-trees we provide bounds on the mixing time.

\end{abstract}
\maketitle

\section{Introduction and main results}
\label{S:motiv}

Let $r_1,r_2\in\R\cup\{-\infty, \infty\}$ with  $r_1<r_2$ and $\nu$
be a Radon measure on $(r_1,r_2)$, i.e., $\nu$
is an inner regular non-negative Borel measure on $(r_1,r_2)$ which is finite on compact sets and positive on any ball.
Then the {\em $\nu$-Brownian motion} on $(r_1,r_2)$ is the unique (up-to $\nu$-equivalence) strong Markov process which is associated with the regular Dirichlet form
\be{e:formR}
   {\mathcal E}(f,g)
 :=
   {\tfrac{1}{2}\int_{(r_1,r_2)}\mathrm{d}\lambda\,f'\cdot g'}
\ee
with domain
\be{e:domainR}
   {\mathcal D}({\mathcal E})  := \big\{f\in L^2(\nu)\cap {\mathcal A_{\R}} : f'\in L^2(\lambda) \big\}
\ee
where $\lambda$ denotes Lebesgue measure and ${\mathcal A_{\R}}$ is the space of absolutely continuous functions that vanish at {regular boundary points}. As usual, we call the left boundary point $r_1$ {\em regular} if it is finite and there exists a point $x\in(r_1,r_2)$
with $\nu(r_1,x)<\infty$. Regularity of the right boundary point, $r_{2},$ is defined in the same way. If $\nu=\lambda$ we obtain {\em standard Brownian motion} while a general $\nu$ plays the r\^{o}le of  the {\em speed measure}. The goal of this paper is to extend this construction of Brownian motion to locally compact $\R$-trees.

In \cite{KumagaiSturm05} a sufficient condition is given to construct non-trivial diffusion processes on a locally compact metric measure space. These processes are associated with local
regular Dirichlet forms which are obtained as suitable limits of approximating
non-local Dirichlet forms. On self-similar sets which can be approximated by an increasing set $(V_m)_{m\in\mathbb{N}}$ diffusions have been studied from a probabilistic and analytic point of view. For example, \cite{Kusuoka1987,Goldstein1987,BarlowPerkins1988,Lindstrom1990} consider random walks on $V_m$ and construct  Brownian motion as the scaling limit. From an analytical point of view this corresponds to constructing the {\em Laplace operator} as the limit of the difference operators corresponding to the approximating random walks.

Tree-like objects have been studied this way as well.  An approximation scheme of the Brownian continuum random tree was exploited in \cite{Kre95}.  The notion of finite resistance forms was introduced in \cite{Kigami95} and these approximating forms yield a regular Dirichlet form on complete, locally compact $\R$-trees. More recently in \cite{Cro08} and \cite{Cro10} scaling limits of simple random walks on random discrete trees have been shown to converge  to Brownian motion on limiting compact $\R$-trees. In a couple of instances diffusions have been constructed using the specific structure of the given $\R$-tree (\cite{DJ93,Eva00}). In \cite{Eva00} the ''richest'' $\mathbb{R}$-tree is considered and a particular diffusion is constructed such that the height process (with respect to a distinguished root) is a standard one-dimensional Brownian motion which in any branch point chooses a direction according to a measure prescribed on the leaves.

The main purpose of the paper is to provide an explicit description of the Dirichlet form of Brownian motion on a given locally compact $\R$-tree without requiring an approximation scheme.  Thus providing a unifying theory from which various properties of the process can be easily read off. Towards this, we imitate the construction of Brownian motion on the real line via Dirichlet forms by  exploiting  the one-dimensional structure of the skeleton of the $\R$-tree.   The first step lies in capturing the key ingredients, namely the length measure and a notion of a gradient (Proposition \ref{P:grad}). Given these ingredients one can then define a bilinear form  similar to  the real line construction. The second step  is then to show that the above bilinear form is a  regular Dirichlet form (Proposition~\ref{P:00} and Proposition~\ref{L:04}) to  ensure the  existence (Theorem~\ref{T:01}) of  a Markov process.  In Proposition~\ref{P:prop} we obtain the characterizing identities for the  occupation measure and hitting probabilities to conclude that the Markov process so constructed is indeed the desired Brownian motion.

On complete and locally compact $\R$-trees  the Brownian motions constructed this way  are the diffusions associated with the finite resistance form introduced in \cite{Kigami95} (see Remark~\ref{Rem:03}).  As we will show in Section~\ref{s:BMdrift}  it covers all the examples of Brownian motions on particular $\R$-trees which can be found in the literature, (See Example \ref{Exp:02} and Example \ref{Exp:06}), and  can also be easily adapted  to construct diffusions with a drift as well.   Furthermore,  we are able to  provide geometric conditions under which the {Brownian motion} is recurrent and transient (Theorems~\ref{T:04} and~\ref{T:trareha}). An interesting application of this result (See Example \ref{karrayt}) generalizes the results shown for random walks on discrete trees in~\cite{Lyo90}.  Bounds on eigenvalues and mixing times (Theorem~\ref{C:mix}),  and various properties of random walks on discrete trees (Theorem~\ref{nashwillconverse}) are obtained for generic $\R$-trees. Thus highlighting the advantages of having an explicit limiting Dirichlet form along with an explicit description of its domain. \sm

We begin by stating some preliminaries in Subsection~\ref{prelim} which will be followed by statements of our main results in Subsection~\ref{mainresults}.

 \subsection{Set-up for Brownian motion on $\R$-tree} \label{prelim}
 In this subsection we discuss preliminaries that are required for constructing Brownian motion on $\R$ trees.

 \medskip

 {\bf $\R$-Tree:} A metric space $(T,r)$ is said to be a {\em real tree }($\R$- {\em tree}) if it
satisfies the following axioms.

\begin{itemize}
\item[{}]{\bf Axiom~1 (Unique geodesic) } For all ${u},{v}\in T$
there exists a unique isometric embedding
  $\phi_{{u},{v}}:[0,{r}({u},{v})]\to T$ such that $\phi_{{u},{v}}(0)={u}$
and $\phi_{{u},{v}}({r}({u},{v}))={v}$.\vspace{1mm}
\item[{}]{\bf Axiom~2 (Loop-free) } For every injective
continuous map $\kappa:[0,1]\to T$ one has
$\kappa([0,1])=\phi_{\kappa(0),\kappa(1)}([0,r(\kappa(0),\kappa(1))])$.
\end{itemize}\sm

Axiom~1 states that there is a unique ``unit speed'' path between
any two points, whereas Axiom~2 then implies  that the image
of any injective path connecting two points coincides with
the image of the unique unit speed path. Consequently any injective path between two points  can be re-parameterized
to become the unit speed path.  Thus, Axiom~1
is satisfied by many other spaces such as $\R^d$
with the usual metric, whereas Axiom~2 expresses the
property of ``tree-ness'' and is only satisfied
by $\R^d$ when $d=1$. We refer the reader to \cite{Dre84, DreMouTer96, DreTer96,
Ter97, MR2003e:20029} for background on $\R$-trees.

For $a,b\in T$, let
\begin{equation}\label{EPWarc}
   [a,b]\,:=\phi_{a,b}(\,[0,r(a,b)]\,)\quad \mbox{and}\quad
   ]a,b[\,:=\phi_{a,b}(\,]0,r(a,b)[\,)
\ee
be the unique closed and open,
respectively, {\em arc}\index{arc} between them.
An immediate consequence of both axioms together is that
real trees are $0$-hyperbolic. For a
given real tree $(T,r)$ and for all $x,a,b\in T$, this implies that there exists a unique
point $c(a,b,x)\in T$ such that
\be{e:branch}
   [a,x]\cap[a,b]=[a,c(a,b,x)].
\ee
The point $c(a,b,x)$ also satisfies $[b,x]\cap[b,a]=[b,c(a,b,x)]$ and $[x,a]\cap[x,b]=[x,c(a,b,x)]$ (see, for example, Lemma~3.20 in \cite{Eva} and compare with Figure 1).
\begin{figure}
\label{Fig:01}
\setlength{\unitlength}{0.7pt}
\begin{picture}(70,60)(-30,-20)
\put(20,32){\mbox{$a$}}
\put(24,7){\mbox{$c(a,b,x)$}}
\put(42,-15){\mbox{$b$}}
\put(-15,-15){\mbox{$x$}}
\put(20,32){\line(0,-1){25}}
\put(20,7){\line(-1,-1){21.1}}
\put(20,7){\line(1,-1){21.1}}
\end{picture}
\caption{}
\end{figure}

In this paper, we will assume that $(T,r)$ is locally compact. By virtue of Lemma~5.7 in \cite{Kigami95}  such $\R$-trees  are separable and by Lemma~5.9   in \cite{Kigami95}  the complete and bounded subsets are compact.

 \medskip

{\bf Length measure:}  We follow \cite{EvaPitWin2006} to introduce the notion of the {\em length measure}
$\lambda^{(T,r)}$ on a separable $\R$-tree $(T,r)$ which extends the
Lebesgue measure on $\R$.
Let  $\mathcal B(T)$ denote the Borel-$\sigma$-algebra of $(T,r)$.
Denote the {\em skeleton} of $(T,r)$ by
\begin{equation}
\label{sce}
   {T}^o:=\bigcup\nolimits_{a,b\in {T}}\,]a,b[.
\ee
Observe that if
${T}^\prime \subset {T}$ is a
dense countable set, then (\ref{sce}) holds with ${T}$ replaced by
${T}^\prime$. In particular, ${T}^o \in {\mathcal B}({T})$ and
${\mathcal B}({T})\big|_{{T}^o}=\sigma(\{]a,b[;\,a,b\in {T}^\prime\})$, where
${\mathcal B}({T})\big|_{{T}^o}:=\{A \cap {T}^o;\,A\in{\mathcal B}({T})\}$.
Hence, there exist a unique $\sigma$-finite measure $\lambda^{(T,r)}$ on $T$, called
{\em length measure}, such that $\lambda^{(T,r)}({T}\setminus {T}^o)=0$ and
\begin{equation}
\label{length}
   \lambda^{(T,r)}(]a,b[)=r(a,b),
\ee
for all $a,b\in T$.
In particular, $\lambda^{(T,r)}$ is the trace onto ${T}^o$ of one-dimensional
Hausdorff measure on $T$. \sm

 \medskip

{\bf Gradient:} We now introduce the notion of weak differentiability and integrability. We will proceed as in~\cite{Eva00}.

Let ${\mathcal C}(T)$ be the space of all real continuous functions on $T$. Consider the subspaces
\begin{equation}
  {\mathcal C}_0(T):=\big\{f\in {\mathcal C}(T)\mbox{ which have compact support}\big\}
\ee
and
\begin{equation}
\label{e:Cinfty}
   {\mathcal C}_{\infty}(T)
 :=
   \big\{f \in {\mathcal C}(T):\,  \forall\, \varepsilon>0\; \exists\, K \mbox{ compact  }\;\forall\, x \in T\setminus K,\;  |f(x)| \leq \varepsilon\big\}
\ee
which is oftne refered to as the space of continuous functions which {\em vanish at infinity}.
\sm

We call a function $f\in{\mathcal C}(T)$
{\em locally absolutely continuous} if and
only if for all $\varepsilon>0$ and all subsets $S\subseteq T$
with $\lambda^{(T,r)}(S)<\infty$
there exists a $\delta=\delta(\varepsilon,S)$ such
that if $[x_1,y_1],...,[x_{n},y_n]\in S$ are disjoint arcs with
$\sum_{i=1}^{n} r(x_i,y_i)<\delta$ then
$\sum_{i=1}^{n}\big|f(x_i)-f(y_i)\big|<\varepsilon$.
Put
\begin{equation}\label{mathcalA}
   {\mathcal A}={\mathcal A}^{(T,r)}
 :=
   \big\{f\in{\mathcal C}(T):\,f\mbox{ is locally absolutely continuous}\big\}.
\ee\sm

In order to define a {\em gradient} of a locally absolutely continuous function, we need the notion of directions on $(T,r)$. For that purpose from now on we fix a point $\rho\in T$ which in the following is referred to as the {\em root}.
Notice that $\rho\in T$ allows us to define a partial order (with respect to $\rho$), $\le_\rho$, on $T$ by saying
that $x\le_\rho y$ for all $x,y\in T$ with $x\in[\rho,y]$.
For all $x,y\in T$ we write
\begin{equation}
\label{e:wedge}
   x\wedge y
 :=
   c(\rho,x,y).
\ee

 The root enables an  orientation sensitive
integration given by
\begin{equation}\label{osi}
\begin{aligned}
   &\int_x^y\lambda^{(T,r)}(\mathrm{d}z)\,g(z)
  \\
  &:=   -\int_{[{x\wedge y},x ]}\lambda^{(T,r)}(\mathrm{d}z)\,g(z)+\int_{[{x\wedge y},y]}\lambda^{(T,r)}(\mathrm{d}z)\,g(z),
\end{aligned}
\ee
for all $x,y\in T$.

The definition of the gradient is then based on the following observation.
\begin{proposition}
Let $f\in\mathcal A$. \label{P:grad}
There exists a unique (up to $\lambda^{(T,r)}$-zero sets) function
$g\in L_{\mathrm{loc}}^1(\lambda^{(T,r)})$ such that
\begin{equation}\label{con.1}
   f(y)-f(x)
 =
   \int_x^y\lambda^{(T,r)}(\mathrm{d}z)\,g(z),
\ee
for all $x,y\in T$. Moreover, $g$ is already uniquely determined  (up to $\lambda^{(T,r)}$-zero sets) if we only require (\ref{con.1}) to hold for all $x,y\in T$ with $x\in[\rho,y]$.
\end{proposition}\sm

\begin{definition}[Gradient] The gradient, $\nabla f=\nabla^{(T,r,\rho)} f,$ of $f\in{\mathcal A}$ is the unique up to $\lambda^{(T,r)}$-zero sets function $g$ which satisfies (\ref{con.1}) for all $x,y\in T$.
\label{Def:01}
\end{definition}\sm

{
\begin{remark}[Dependence on the choice of the root] \rm Fix a separable $\mathbb{R}$-tree $(T,r)$.
 Notice that the gradient $\nabla f$ of a function $f\in{\mathcal A}$ depends on the particular choice of the root $\rho\in T$ (compare Examples~\ref{Exp:04} and~\ref{Exp:01}). It is, however, easy to verify that  for each $\rho\in T$ there exists a $\{-1,1\}$-valued function $\sigma^\rho:T\to\{-1,1\}$ and for all $f\in{\mathcal A}$ a function $g^f:T\to\mathbb{R}$ such that $\nabla f$ is of the following form:
 \begin{equation}
 \label{e:formgrad}
   \nabla f=\sigma^\rho\cdot g^f.
 \end{equation}
 \label{Rem:06}
\end{remark}\sm
}

 \medskip
{\bf The Dirichlet form: } {Let $(T,r)$ be a separable $\mathbb{R}$-tree and $\nu$ a Borel measure on $(T,{\mathcal B}(T))$. Denote, as usual,  by $L^2(\nu)$ the space of Borel-measurable functions on $T$ which are square integrable with respect to $\nu$. As usual, for $f,g\in L^2(\nu)$ we denote by
\begin{equation}
\label{e:inner}
   \big(f,g\big)_\nu:=\int\mathrm{d}\nu\,f\cdot g
\ee
the {\em inner product} of $f$ and $g$ with respect to $\nu$.

Put}
\begin{equation}
\label{e:F}
   {\mathcal F}:=
   \big\{f \in {\mathcal A}:\, \nabla f\in L^2(\lambda^{(T,r)})\big\},
\ee
and consider the domain
\begin{equation}
\label{domainp}
   {\mathcal D}(\mathcal E)
 :=
   {\mathcal F} \cap L^2(\nu) \cap   {\mathcal C}_{\infty}(T)
\ee
together with the bilinear form
\begin{equation}\label{con.2p}
\begin{aligned}
   {\mathcal E}(f,g)
 &:=
   \frac{1}{2}\int\lambda^{(T,r)}(\mathrm{d}z)\nabla f(z)
   \nabla g(z)
\end{aligned}
\ee
for all $f,g\in{\mathcal D}({\mathcal E})$.
{Notice that this bilinear form is independent of the particular choice of $\rho$ by Remark~\ref{Rem:06}.}
\sm

\subsection{Main Results}
\label{mainresults}
%
In this subsection we shall state all our main results.
Unless stated otherwise throughout the paper we shall assume that
\begin{itemize}
\item[(A1)] $(T,r)$ is  a locally compact $\R$-tree.
\item[(A2)] $\nu$ is a Radon measure on $(T,{\mathcal B}(T))$, i.e., $\nu$ is finite on compact sets and positive on
any open ball
\begin{equation}
\label{e:ball}
   B(x,\varepsilon)
 :=
   \big\{x'\in T:\,r(x,x')<\varepsilon\big\}
\end{equation}
with $x\in T$ and $\varepsilon>0$.
\end{itemize}

Our first main result is the following:

\begin{theorem}[Brownian motion on $(T,r,\nu)$]
Assume (A1) and (A2). There exists a {unique (up to $\nu$-equivalence)}
continuous $\nu$-symmetric strong Markov process
$B=((B_t)_{t\ge 0},({\mathbf P}^x)_{x\in T})$ on $(T,r)$
whose Dirichlet form is
$({\mathcal E}, {\mathcal D}({\mathcal E}))$.
\label{T:01}\end{theorem}\sm

This leads to the following definition.
\begin{definition}[Brownian motion] The $\nu$-symmetric strong Markov process
$B=((B_t)_{t\ge 0},({\mathbf P}^x)_{x\in T})$ on $(T,r)$
associated with the Dirichlet form $({\mathcal E}, {\mathcal D}({\mathcal E}))$ is called $\nu$-Brownian motion on the $\R$-tree $(T,r)$.
\label{Def:05}
\end{definition}\sm

{
\begin{remark}[The role of $\nu$]\em $\nu$-Brownian motion on $(T,r)$ can be thought of as a diffusion on $(T,r)$ which is on {\em natural scale} and has
{\em speed measure} $\nu$. With a slight arbitrament we shall refer to $B$ as the {\em standard Brownian motion} if $\nu$ equals the Hausdorff measure on $(T,r)$.
\label{Rem:07}
\hfill$\qed$
\end{remark}\sm
}

{
\begin{remark}[\bf Kigami's resistance form on dendrites] \rm Let $(T,r)$ be a locally compact and complete $\R$-tree and $\nu$ a Radon measure on $(T,{\mathcal B}(T))$. Let furthermore $(V_m)_{m\in\mathbb{N}}$ be an increasing
and compatible (in the sense of Definition~0.2 in \cite{Kigami95}) family of finite subsets of $T$ such that $V^\ast:=\cup_{m\in\mathbb{N}}V_m$ is
countable and dense. For each $m\in\mathbb{N}$ and $x,y\in V_m$, let
$x\sim y$ whenever $]x,y[\cap V_m=\emptyset$, and put for all $f,g:V_m\to\R$
\begin{equation}
\label{Kigami:m}
   {\mathcal E}_m(f,g)
 :=
   \tfrac{1}{2}\sum\nolimits_{x,y\in V_m;x\sim y}\tfrac{(f(x)-f(y))(g(x)-g(y))}{r(x,y)}.
\end{equation}

In \cite{Kigami95} the bilinear form
\begin{equation}
\label{Kigami:form}
   {\mathcal E}^{\mathrm{Kigami}}(f,g)
 :=
   \lim_{m\to\infty}{\mathcal E}_m\big(f\big|_{V_m},g\big|_{V_m}\big)
\ee
with domain
\begin{equation}
\label{Kigami:label}
   {\mathcal F}^{\mathrm{Kigami}}
 :=
   \big\{f:V^\ast\to\mathbb{R}:\,\mbox{limit on r.h.s. of (\ref{Kigami:form}) exists}\big\}
\ee
is studied.

Put
\begin{equation}
\label{Kigami:labelD}
{\mathcal D}\big({\mathcal E}^{\mathrm{Kigami}}\big) := \overline{ {\mathcal F}^{\mathrm{Kigami}} \cap  {\mathcal C}_0(T)} ^{{\mathcal E}^{\mathrm{Kigami}}_{1}},
\ee
where the closure is with respect  to the ${\mathcal E}^{\mathrm{Kigami}}_1$-norm given by
\begin{equation}
\label{e:024}
   {\mathcal E}^{\mathrm{Kigami}}_{1}(f,g):={\mathcal E}^{\mathrm{Kigami}}(f,g) + (f,g)_\nu.
\ee
It is (partily) shown in
Theorem~5.4 in \cite{Kigami95} that $({\mathcal E}^{\mathrm{Kigami}},{\mathcal D}({\mathcal E}^{\mathrm{Kigami}}))$
is a regular Dirichlet form.
Notice that Theorem~5.4 in \cite{Kigami95} actually only assumes the measure $\nu$ to be a $\sigma$-finite Borel measure that charges all open sets,
and defines the domain to be  ${\mathcal F}^{\mathrm{Kigami}} \cap L^{2}(\nu)$.
In order to ensure regularity, however, one needs to indeed close the  Kigami suggested domain
$ {\mathcal F}^{\mathrm{Kigami}} \cap{\mathcal C}_0(T)$ with respect to the  ${\mathcal E}^{\mathrm{Kigami}}_1$-norm. Moreover, regularity forces
$\nu$ to be a Radon measure; a fact which is used in Kigami's proof.

We will prove  in Remark~\ref{Rem:05} that $({\mathcal E},{\mathcal D}({\mathcal E}))$ agrees with Kigami's form on complete locally compact $\R$-trees. Note that our set-up is slightly more general (do not require completness) and the notion of a gradient at hand provides an explicit description of the form. 
\label{Rem:03}
\hfill$\qed$
\end{remark}\sm
}

For all closed $A\subseteq T$, let
\begin{equation}\label{tauA}
   \tau_A
 :=
   \inf\big\{t  > 0:\,B_t\in A\big\}
\ee
denote the {\em first hitting time} of the set $A$. In particular, put $\tau_A:=\infty$ if $\cup_{t > 0}\{B_t\}\subseteq T\setminus A$. Abbreviate $\tau_x:=\tau_{\{x\}}$, $x\in T$.
\begin{definition}[Recurrence/transience]
The   $\nu$-Brownian motion $B$ on the $\R$-tree $(T,r)$
is called {\rm transient} iff
\begin{equation}\label{e:013}
 \int_0^\infty\mathrm{d}u\,\mathbf P^{\rho}\{B_u \in K \}<\infty,
\ee
for all compact subsets $K\subseteq T$.
Otherwise, the   $\nu$-Brownian motion on the $\R$-tree $(T,r)$ is called {\rm recurrent}.

We say that a recurrent $\nu$-Brownian motion on $(T,r)$ is {\rm null-recurrent} if there exists a $y\in T$ such that $\mathbf E^x[\tau_y]=\infty$, for some $x\in T$, and {\rm positive recurrent} otherwise.
\label{Def:04}
\end{definition}\sm

{
\begin{remark}\rm As we will observe in Lemma~\ref{L:02}, $B$ has a $\nu$-symmetric transition densities $p_t(x,y)$ with respect
to $\nu$ such that $p_t(x;y) > 0$ for all $x, y\in T$. Consequently, in the terminology
of \cite{FukushimaOshimaTakeda1994}, $B$ is irreducible. Therefore, by Lemma~1.6.4 of \cite{FukushimaOshimaTakeda1994}, $B$ is either
transient or recurrent.
\hfill$\qed$
\end{remark}\sm
}

To justify the name ``Brownian motion'', we
next  verify that $\nu$-Brownian motion on $(T,r)$ satisfies  the  characterizations of Brownian motion  (known  on $\R$).

\begin{proposition}[Occupation time measure]
Assume (A1) and (A2). Let $B=((B_t)_{t\ge 0},({\mathbf P}^x)_{x\in T})$ be the continuous $\nu$-symmetric strong Markov process\label{P:prop}
 on $(T,r)$
whose Dirichlet form is $({\mathcal E}, {\mathcal D}({\mathcal E}))$.
Then the following hold:
\begin{itemize}
\item[(i)]
For all $a,b,x\in T$ such that $\mathbb{P}^x\{\tau_a\wedge \tau_b<\infty\}=1$,
\be{Xhit}
   \mathbf P^x\big\{\tau_a<\tau_b\big\}
 =
   \frac{r(c(x,a,b),b)}{r(a,b)}.
\ee
\item[(ii)] Assume furthermore that the measure $\R$-tree $(T,r,\nu)$ is  such that the $\nu$-Brownian motion $(T,r)$ is recurrent.
For all $b,x\in T$ and bounded measurable $f$,
\be{Xocc}
   \mathbf{E}^x\big[\int_0^{\tau_b}\mathrm{d}t\,f(B_s)\big]=2\int_{T}\nu(\mathrm{d}y)\,r\big(c(y,x,b),b\big)f(y).
\ee
\end{itemize}
\end{proposition}\sm

\begin{remark}\rm
Proposition~\ref{P:prop} has been verified for the $\nu$-Brownian motion on the Brownian CRT for two particular choices of $\nu$ in \cite{Kre95} and \cite{Cro08}
(compare also Example~\ref{Exp:02}).
\hfill$\qed$
\label{Rem:04}
\end{remark}\sm

A second goal of this paper is to give a criterion for the $\nu$-Brownian motion on $(T,r)$ to be recurrent or transient. For a subset $A\subseteq T$, denote by
\begin{equation}\label{e:diam}
    \mathrm{diam}^{(T,r)}(A)
 :=
    \sup\big\{r(x,y):\,x,y\in A\big\}
\ee
its diameter. For bounded trees (i.e.\ those with finite diameter) recurrence and transience depends on whether or not $(T,r)$ is compact.
\begin{theorem}[Recurrence/transience on bounded trees] Let $(T,r)$ be a  bounded $\R$-tree. Assume (A1) and (A2).
 \begin{itemize}
 \item[(i)] If $T$ is compact then $\nu$-Brownian motion on $(T,r)$ is positive recurrent.
 \item[(ii)] If $T$ is not compact then $\nu$-Brownian motion on $(T,r)$ is transient.
 \end{itemize}
\label{T:04}
\end{theorem}\sm

Obviously, a bounded { and locally compact} $\R$-tree is complete if and only it is compact. Therefore {Theorem~\ref{T:04} states} that the $\nu$-Brownian motion on a bounded locally compact $\R$-tree is positive recurrent if the tree is complete and transient if the tree is incomplete.
In the case of compact $\R$-trees we can also give bounds on the mixing time.

\begin{theorem}[Mixing time] {Let $(T,r)$ be a compact $\R$-tree, and $\nu$ a Radon measure on $(T,{\mathcal B}(T))$.} Let {$(P_t)_{t\ge 0}$}  be the semi-group associated with {the} $\nu$-Brownian motion on $(T,r)$.
\label{C:mix}
If $\nu'$ is a probability measure on $(T,{\mathcal B}(T))$ with $\nu'\ll\nu$ such that
$\tfrac{\mathrm d\nu'}{\mathrm d\nu}\in L^1(\nu')$, then for all $t{\ge} 0$,
\begin{equation}\label{e:mix.1}
\begin{aligned}
   &\big\|\nu'P_t-\big(\nu(T)\big)^{-1}\nu\big\|_{\mathrm{TV}}
  \\
 &\leq
   \left(1+\nu(T)\sqrt{(\mathbf{1}_T,\tfrac{\mathrm d\nu'}{\mathrm d\nu})_{\nu'}} \,\,\right)\cdot\mathrm e^{-t/2\mathrm{diam}^{(T,r)}(T)\nu(T)},
\end{aligned}
\ee
where $\|\boldsymbol{\cdot}\|_{\mathrm{TV}}$ denotes the total variation norm.
\end{theorem}\sm

We next state a geometric criterion for recurrence versus
transience for {\em unbounded} trees.
As a preparation we introduce the space of ends at infinity and recall the notion of {the} {\em Hausdorff dimension}.\sm

 \medskip

{\bf The space of ends at infinity $(E_{\infty},\bar{r})$: }
If $(T,r)$ is unbounded then there exists an isometric embedding {$\phi$ from $\R_+:=[0,\infty)$ into $T$ with $\phi(0)=\rho$}.
In the following we refer to each {such isometry } as an {\em end at infinity}, and let
\begin{equation}
\label{Einfty}
   E_\infty
 :=
   \mbox{ set of all ends at infinity.}
\end{equation}

Recall that $\rho\in T$ is a fixed root which allows to define a partial order $\le_\rho$ on $T$ by saying
that $x\le_\rho y$ for all $x,y\in T$ with $x\in[\rho,y]$.
This partial order $\le_\rho$ extends to a partial order on $T\cup E_\infty$ by letting for each {$x\in T$ and $y\in E_\infty$},
$x\le_\rho y$ if and only if {$x\in y(\R_+)$. Further for $x,y \in E_{\infty}$, $x\le_\rho y$ if and only if $x=y.$ Each pair $x,y\in T\cup E_\infty$ has then a well-defined {\em greatest common lower bound}
\begin{equation}
\label{brach}
   x\wedge y
 =
   x\wedge_\rho y\in T\cup E_\infty.
\ee

We equip $E_\infty$ with the metric $\overline r(\boldsymbol{\cdot},\boldsymbol{\cdot})$ defined by
\begin{equation}\label{e:barr}
   \bar{r}(x,y)=\bar{r}_\rho(x,y)
 :=
   1\wedge \frac{1}{r(\rho,x\wedge y)},
\ee
for all $x,y\in E_\infty$.

It is not difficult to see that $(E_\infty,\bar{r})$ is ultra-metric.
Hence by Theorem~3.38 in \cite{Eva} for all subsets $E'\subseteq E_\infty$ there is a (smallest) $\R$-tree $(T',r')$ with $E'\subseteq T'$ and such that $\bar{r}(x,y)=r'(x,y)$ for all $x,y\in E'$. We will refer to this smallest $\R$-tree as the $\R$-tree {\em spanned by $(E',\bar{r})$} and denote it by \begin{equation}
\label{e:023}
   \mathrm{span}(E',\bar{r}).
\ee \sm
It is easy to see that $\mathrm{span}(E',\bar{r})$  is a  compact $\R$-tree which has the same tree-topology as $(T,r)$ {outside $B(\rho,1)$}.
\medskip

{\bf  Hausdorff dimension of $E_{\infty}$: }
 For all $\alpha\ge 0$,
the {\em $\alpha$-dimensional Hausdorff measure} ${\mathcal H}^\alpha$ on $(E_\infty,{\mathcal B}(E_\infty))$ is defined as follows:
for all $A\in{\mathcal B}(E_{\infty})$, let
\begin{equation}\label{e:inha}
\begin{aligned}
   &{\mathcal H}^\alpha(A)
   \\
 &:=
   \lim_{\varepsilon\downarrow 0}\,
   \inf\big\{\sum_{i\ge 1}\big(\mathrm{diam}^{(E_{\infty}, \overline r)}(E_i)\big)^\alpha:\;
   \bigcup_{i\ge 1}E_i\supseteq A,\,\mathrm{diam}^{(E_{\infty}, \overline r))}(E_i)\le\varepsilon\big\}.
\end{aligned}
\ee

The {\em Hausdorff dimension} of a subset $A\in{\mathcal B}(E_{\infty})$ is then defined as
\begin{equation}\label{dim}
\begin{aligned}
   \mathrm{dim}_{\mathrm{H}}^{(E_{\infty},\overline r)}(A)
 &:=
   \inf\big\{\alpha\ge 0:\,{\mathcal H}^\alpha(A)=0\big\}
  \\
 &=
   \sup\big\{\alpha\ge 0:\,{\mathcal H}^\alpha(A)=\infty\big\}.
\end{aligned}
\ee

\begin{remark}\rm
Note that $\dim_{\mathrm{H}}^{(E_\infty,\overline r)}(E_\infty)$ does not depend on the particular choice of $\rho\in T$.
\label{Rem:08}
\hfill$\qed$
\end{remark}\sm

We are now ready to  state a geometric criterion for recurrence and transience on trees with ends at infinity.
\begin{theorem}[Recurrence/transience on unbounded trees] {Let $(T,r)$  be a locally compact $\R$-tree such that $E_\infty\not = \emptyset$, and $\nu$ a Radon measure on $(T,{\mathcal B}(T))$.}
\label{T:trareha}
\begin{itemize}
\item[(i)] If $(T,r)$ is complete and ${\mathcal H}^1$ is a finite measure,
 then
the $\nu$-Brownian motion on $(T,r)$ is recurrent.
\item[(ii)] If
$\dim_{\mathrm{H}}^{(E_\infty,\overline r)}(E_\infty)>1$ or $(T,r)$ is incomplete, then
the $\nu$-Brownian motion on $(T,r)$ is transient.
\end{itemize}
\end{theorem}\sm


The following example illustrates an application of Theorems~\ref{T:04} and~\ref{T:trareha}
{suggesting a duality between bounded and unbounded trees}.

\begin{example}[The $k$-ary tree] \rm We want to illustrate the theorem with the example of symmetric trees. Fix $k\ge2$ and $c>0$ and
\label{karrayt}
We want to illustrate the theorem with the example of symmetric trees. Fix $k\ge2$ and $c>0$ and
let $(T,r)$ be the
following locally compact
$\R$-tree uniquely characterized as follows:
\begin{itemize}
\item
There is a root $\rho\in T$.
\item A point $x\in T$ is a {\em branch point}, i.e., $T\setminus\{x\}$ consists of more than $2$
connected components, if and only if
$r(\rho,x)=\sum_{l=0}^mc^{l}$, for some $m\in\N \cup \{0\}$.
\item All branch points are of {\em degree} $k+1$, i.e., $T\setminus\{x\}$ consists of $k+1$
connected components.
\end{itemize}
It is easy to check that for all choices of $c>0$, the length measure $\lambda^{(T,r)}$ is Radon.
Hence $\lambda^{(T,r)}$-Brownian motion on $(T,r)$
exists by Theorem~\ref{T:01}.

Since
\begin{equation}\label{e:diamk}
   \mathrm{diam}^{(T,r)}(T)
 =
    \sum\nolimits_{l\in\N}c^l
    \left\{\begin{array}{cc}<\infty, & \mbox{ if }c<1,\\ =\infty, & \mbox{ if }c\ge 1,\end{array} \right.
\ee
the tree is bounded iff $c<1$. We discuss bounded and unbounded trees separately.

Assume first that $c<1$. By construction, $(T,r)$ is not compact and hence $\lambda^{(T,r)}$-Brownian motion is transient.
Notice that since
\begin{equation}
   \lambda^{(T,r)}(T)
 =
    \sum\nolimits_{l\in\N}k^l\cdot c^l
   \left\{\begin{array}{cc}<\infty, & \mbox{ if }c<\frac{1}{k},\\ &\\ =\infty, & \mbox{ if }c\ge \frac{1}{k},\end{array} \right.
\ee
$\lambda^{(T,r)}$-Brownian motion exists also on the completion $(\bar{T},r)$ of $(T,r)$ in the case $c\in(0,\frac{1}{k})$.
Since a complete, bounded, and locally compact $\R$-tree is compact,
$\lambda^{(T,r)}$-Brownian motion on $(\bar{T},r)$ is positive recurrent by Theorem~\ref{T:04}. Note that $\lambda^{(T,r)}$-Brownian motion on $(T,r)$ versus $(\bar{T},r)$
differ in their behaviour on the boundary $\partial T:=\bar{T}\setminus T$. While the first process gets killed on $\partial T$, the second gets reflected at $\partial T$.

Assume next that $c\ge 1$. An easy calculation shows that $\mathrm{dim_H}^{(E_\infty,\bar{r})}(E_\infty)=\log_c(k)$, and hence the
$\lambda^{(T,r)}$-Brownian motion is recurrent if
$c>k$ and transient if $c<k$ by Theorem~\ref{T:trareha}. The latter has been shown for random walks in~\cite{Lyo90}.
{Further, when $c=k$ it can be easily verified that the Hausdorff measure of $E_\infty$ is bounded by $2k<\infty$, which implies that
$\lambda^{(T,r)}$-Brownian motion is recurrent at the critical value $c=k$.}
\hfill$\qed$
\end{example}\sm

We conclude this section with a result that shows how the $\lambda^{(T,r)}$-Brownian motion on locally  compact $\R$-trees which are spanned by their ends at infinity
can be used to  decide whether or not random walks, simple or weighted, on graph-theoretical trees are recurrent.\sm

 \medskip

{\bf Graph-Theoretical Trees:}
Consider a non-empty countable set~$V$ and a family of non-negative weights $\{r_{\{x,y\}};\,x,y\in V\}$ such that $(V,E)$ is a locally finite graph-theoretical tree, where $E:=\{\{x,y\}\mbox{ with }x,y\in V;r_{\{x,y\}}>0\}$. In the following we refer to $(V,\{r_{\{x,y\}};\,x,y\in V\})$ as a {\em weighted, discrete tree}. A Markov chain $X=(X_n)_{n\in\N_0}$ on the weighted, discrete tree $(V,\{r_{\{x,y\}};\,x,y\in V\})$
allows transitions between any neighboring points $x,y\in T$ with $r_{\{x,y\}}>0$ and probabilities proportional to the {\em conductance} $c_{\{x,y\}}:=(r_{\{x,y\}})^{-1}$.

Call an infinite sequence $(x_n)_{n\in\N_0}$ of distinct vertices in $V$ with $x_0=\rho$ and $r_{\{x_n,x_{n+1}\}}>0$ for all $n\in\N$ a {\em direction} in $(V,\{r_{\{x,y\}};\,x,y\in V\})$, and denote similar to (\ref{Einfty}) by $\tilde{E}_\infty$ the set of all directions. Let for any two directions $x=(x_n)_{n\in\N}$ and $y=(y_n)_{n\in\N}$,
$k(x,y)$ denote the last index $k\in\N\cup\{\infty\}$ for which $x_k=y_k$, and define
$x\wedge y:=x_{k(x,y)}$ if $k(x,y)\in\mathbb{N}$, and $x\wedge y:=x\in\tilde{E}_\infty$ if $k(x,y)=\infty$.
Recall from (\ref{e:barr}) the metric $\bar{r}$, and define in a similar way a metric $\tilde{r}$ on $\tilde{E}_\infty$ by letting for all $x,y\in\tilde{E}_\infty$, $\tilde{r}(x,y):=(r(x\wedge y,\rho))^{-1}\wedge 1$.

\begin{theorem}[Recurrence versus transience of random walks on trees]
Let $(V,\{r_{\{x,y\}};\,x,y\in V\})$ be a weighted discrete tree such that for all directions $x=(x_n)_{n\in\N}$, $\sum_{n\in\N}r_{\{x_n,x_{n+1}\}}=\infty$.
Then the random walk $X$ is recurrent if \label{nashwillconverse}
{${\mathcal H}^1$ is  a finite measure on $(\tilde{E}_\infty,{\mathcal B}(\tilde{E}_\infty))$} and
transient if $\dim_H(\tilde{E}_\infty,\tilde r)>1$.
\end{theorem}\sm

\subsection{Outline.} The rest of the paper is organized as follows. In Section~\ref{S:Dirichlet} we introduce the Dirichlet space associated with the Brownian motion.
In Section~\ref{S:capacities} we recall the relevant potential theory and apply it to give explicit
expressions for the
capacities and Green kernels associated with the Dirichlet form. In Section~\ref{S:existence} we
prove the existence of a strong
Markov process with continuous paths which is associated with the Dirichlet form.
In Section~\ref{S:compact} we study the basic long-term behavior
for Brownian motions on locally-compact and bounded $\R$-trees. More precisely,  we prove Theorem~\ref{T:04}
and give in the recurrent case lower and upper bounds for the principle eigenvalue and the spectral gap.  We prove Theorem~\ref{T:trareha} in Section~\ref{S:transinfty}.
In Section~\ref{Sub:contdisc} we recover and generalize for $\R$-trees which can be spanned by their ends at infinity results for the embedded
random walks as known from~\cite{Lyo90}. In particular, we give the proof of Theorem~\ref{nashwillconverse}. Finally in Section~\ref{s:BMdrift} we discuss examples in the literature and diffusions that are not on natural scale.

\subsection*{Acknowledgments} Michael Eckhoff passed away during the completion of this work. The core theme and ideas in the paper are in part due to him. Further, several key estimates and ideas from Dirichlet form theory were brought to our notice by him.  Our deepest condolences to his family.

 We would like to thank David Aldous for proposing a problem that initiated this project and Zhen-Qing Chen, Steve Evans, Wolfgang L\"ohr and Christoph Schumacher  for helpful discussions. Thanks are due to an anonymous referee, whose earlier  detailed report helped us in preparing this revised version of the article.  
 
 Siva Athreya was supported in part by a CSIR Grant in Aid scheme and Homi Bhaba Fellowship. Anita Winter was supported in part at the Technion by a fellowship from the Aly Kaufman Foundation.

\section{The Dirichlet space}
\label{S:Dirichlet}
Fix $(T,r)$ to be a locally compact $\R$-tree and $\nu$ a Radon measure on $(T,{\mathcal B}(T))$.
In this section we construct the Dirichlet space (to be) associated with the $\nu$-Brownian motion.
In Subsection~\ref{Sub:proof21}, we begin with giving the proof of Proposition \ref{P:grad}. In Subsection~\ref{Sub:form}
we verify that $({\mathcal E},{\mathcal D}({\mathcal E}))$ from (\ref{domainp}) and (\ref{con.2p}) is indeed a Dirichlet form.

\subsection{The gradient (Proof of Proposition~\ref{P:grad})}
\label{Sub:proof21}
\begin{proof}[Proof of Proposition~\ref{P:grad}]
Fix a root $\rho\in T$, and $x,y\in T$.

Assume for the moment that $x,y\in T$ are such that $x\in[\rho,y]$.  By Axiom~1, there is a unique isometric embedding $\phi_{x,y}:[0,r(x,y)]\to[x,y]$.  Fix $f\in{\mathcal A}$, and define the function $F_{x,y}:[0,r(x,y)]\to\R$ by
$F_{x,y}:=f\circ\phi_{x,y}$. Since $\phi_{x,y}$ is an isometry, $F_{x,y}$ is locally absolutely continuous on $\R$.
Hence by standard theory (compare, for example, Theorem~7.5.10 in \cite{AthSun09}),
$F_{x,y}$ is  almost everywhere differentiable,
its derivative $F'_{x,y}$ is Lebesgue integrable and
\begin{equation}\label{grund}
\begin{aligned}
  f(y)-f(x)
 &=
  F_{x,y}(r(x,y))-F_{x,y}(0)
  \\
 &=
  \int_{[0,r(x,y)]}\mathrm{d}t\,F_{x,y}^\prime(t)
  \\
 &=
   \int_{[x,y]}\lambda^{(T,r)}(\mathrm{d}z)\,F_{x,y}'(\phi^{-1}_{x,y}(z))
  \\
 &=
  \int_x^y\lambda^{(T,r)}(\mathrm{d}z)\,F_{x,y}'(\phi^{-1}_{x,y}(z)).
\end{aligned}
\ee

Notice that for all $z\in [x,y]$, we have $F_{x,y}(\phi^{-1}_{x,y}(z))=F_{\rho,y}(\phi^{-1}_{\rho,y}(z))$. Hence, $F_{x,y}'(\phi^{-1}_{x,y}(z))$ does not depend explicitly on $x\in[\rho,y]$. Similarly, for any $y_{1},y_{2}\in T$, $F_{\rho,y_i}(\phi^{-1}_{\rho,y_i})=F_{\rho,y_1\wedge y_2}(\phi^{-1}_{\rho,y_1\wedge y_2})$, for $i=1,2$, on $[\rho,y_1\wedge y_2]$. This implies that for all $y_1,y_2\in T$, $F'_{\rho,y_1}(\phi^{-1}_{\rho,y_1})=F'_{\rho,y_2}(\phi^{-1}_{\rho,y_2})$ on $[\rho,y_1\wedge y_2]$.  Therefore  $F'_{x,y}(\phi^{-1}_{x,y}(z))$  does not depend on the direction given through $[\rho,y]$, and so does not depend on $x,y$.
Consequently, $g:T\rightarrow \R$ given by $g(z):= F_{x,y}'(\phi^{-1}_{x,y}(z))$ when $z \in [x,y]$ satisfies (\ref{con.1}). Local integrability and uniqueness follow by standard measure theoretic arguments. \sm

Let now $x,y\in T$ be arbitrary. Then by what we have shown so far
\begin{equation}\label{grund2}
\begin{aligned}
   &f(y)-f(x)
  \\
 &=
   f(y)-f(\rho)+f(\rho)-f(x)
 \\
 &=
   -{\int_{\rho}^x}\lambda^{(T,r)}(\mathrm{d}z)\,\nabla f(z)+{\int_\rho^y}\lambda^{(T,r)}(\mathrm{d}z)\,\nabla f(z)
  \\
 &=
   \int_x^y\lambda^{(T,r)}(\mathrm{d}z)\,\nabla f(z),
\end{aligned}
\ee
and the claim follows.
\end{proof}\sm

\begin{example}[Distance to a fixed point]\rm Fix $a\in T$, and define $g_a:T\to\R_+$ as
\begin{equation}\label{e:h0}
   g_{a}(x)
 :=
   r\big(a,x\big),
\ee
for all $x\in T$.
Obviously,
$g_{a}$ is
   absolutely continuous. Observe that moving the argument outside the arc $[\rho,a]$ away from the root lets the distance grow at speed $1$, while moving the argument inside the arc $[\rho,a]$ away from the root lets the distance decrease with speed one. We therefore expect that a version of $\nabla g_{a}$ is given by
\begin{equation}\label{e:nablah0}
   \nabla g_{a}(x)
 =
   \mathbf{1}_{T}(x)-2\cdot \mathbf{1}_{[\rho,a]}(x)
\ee
for all $x\in T$.

To see this it is enough to verify (\ref{osi}) for all $x,y\in T$ with $x\in[\rho,y]$. Indeed,
{
\begin{equation}\label{e:yy0}
\begin{aligned}
   &\int^y_x\lambda^{(T,r)}(\mathrm{d}z)\,\big(\mathbf{1}_{T}(z)-2\cdot \mathbf{1}_{[\rho,a]}(z)\big)
  \\
 &=
   r(\rho,y)-r(\rho,x)-2\cdot r(\rho,a\wedge y)+2\cdot r(\rho,a\wedge x)
  \\
 &=
   r(\rho,a)+r(\rho,y)-2\cdot r(\rho,a\wedge y)-r(\rho,a)-r(\rho,x)+2\cdot r(\rho,a\wedge x)
  \\
 &=
   g_a(y)-g_a(x).\mbox{\hfill}\qed
\end{aligned}
\ee}
\label{Exp:04}
\end{example}\sm

\begin{example}[Distance between branch and end point on an arc]\rm Fix $a,b\in T$, and recall the definition of branch points
from (\ref{e:branch}). Define $f_{a,b}:T\to\R_+$ by
\begin{equation}\label{e:h}
   f_{a,b}(x)
 :=
   r\big(c(x,a,b),b\big),
\ee
for all $x\in T$.
Obviously,
$f_{a,b}$ is
   absolutely continuous. Observe that now disturbing the argument outside the arc $[a,b]$ does not change the value of the function while moving the argument away from the root along $[a,a\wedge b]$ and $[a\wedge b,b]$ let the distance grow and decrease, respectively, with speed one. We therefore expect that a version of $\nabla f_{a,b}$ is given by
\begin{equation}\label{e:nablah}
   \nabla f_{a,b}(x)
 =
   \mathbf{1}_{[a,a\wedge b]}(x)-\mathbf{1}_{[a\wedge b,b]}(x)
\ee
for all $x\in T$.

To see this it is enough to verify (\ref{osi}) for all $x,y\in T$ with $x\in[\rho,y]$.
Indeed,\label{Exp:01}
{
\begin{equation}\label{e:yy}
\begin{aligned}
   &\int^y_x\lambda^{(T,r)}(\mathrm{d}z)\,\big(\mathbf{1}_{[a,a\wedge b]}(z)-\mathbf{1}_{[a\wedge b,b]}(z)\big)
  \\
 &=
   \lambda([x,y]\cap[a,a\wedge b])-\lambda([x,y]\cap[b,a\wedge b])
  \\
 &=
   \big(\mathbf{1}\{c(x,a,b)\in[\rho,a]-\mathbf{1}\{c(x,a,b)\in[\rho,b]\}\big)\cdot r\big(c(y,a,b),c(x,a,b)\big)
  \\
 &=
   f_{a,b}(y)-f_{a,b}(x).\mbox{\hfill$\qed$}
\end{aligned}
\ee}
\end{example}\sm

\subsection{The Dirichlet form}
\label{Sub:form}
Let $(T,r)$ be a locally compact $\R$-tree and $\nu$ a {\em Radon measure} {on $(T,{\mathcal B}(T))$}.
Recall from (\ref{domainp}) and (\ref{con.2p}) the bilinear form $({\mathcal E},{\mathcal D}({\mathcal E}))$.
{
\begin{lemma} Fix $a,b\in T$, and recall the function $f_{a,b}$ from Example~\ref{Exp:01}.
Then for all $f\in {\mathcal D}({\mathcal E})$, we have that also $\tilde{f}_{a,b}:=f\cdot f_{a,b}\in  {\mathcal D}({\mathcal E})$.
\label{L:13}
In particular, if $\mathbf{1}_T\in {\mathcal D}({\mathcal E})$ then also $f_{a,b}\in  {\mathcal D}({\mathcal E})$.
\end{lemma}\sm

\begin{proof} By definition,
\begin{equation}\label{e:calcc}
\begin{aligned}
   \nabla\tilde{f}_{a,b}
 &=
   \nabla f\cdot f_{a,b}+f\cdot\big(\mathbf{1}_{[a,a\wedge b]}-\mathbf{1}_{[b,a\wedge b]}\big).
\end{aligned}
\ee

Furthermore
\begin{equation}\label{e:calcc2}
\begin{aligned}
   &\big(\nabla\tilde{f}_{a,b}\big)^2
  \\
 &=
   \big(\nabla f\big)^2\cdot f^2_{a,b}+f^2\cdot\mathbf{1}_{[a,b]}+2\cdot f\cdot\nabla f\cdot f_{a,b}
   \cdot\big(\mathbf{1}_{[a,a\wedge b]}-\mathbf{1}_{[b,a \wedge b]}\big)
  \\
 &\le
   \big(\nabla f\big)^2\cdot r^2(a,b)+f^2\cdot\mathbf{1}_{[a,b]}+2r(a,b)\cdot {|f|}\cdot\big|\nabla f\big|
   \cdot\mathbf{1}_{[a,b]},
\end{aligned}
\ee
which implies that\label{Exp:05}
\begin{equation}\label{dirform}
\begin{aligned}
   &{\mathcal E}\big(\tilde{f}_{a,b},\tilde{f}_{a,b}\big)
   \\
 &\le
   r^2(a,b){\mathcal E}(f,f)+\frac{1}{2}\int_{[a,b]}\lambda^{(T,r)}(\mathrm{d}z)\,\big(f^2+2r(a,b)|f|\cdot|\nabla f|\big)
  \\
 &\le
   r^2(a,b){\mathcal E}(f,f)+\frac{1}{2}\int_{[a,b]}\lambda^{(T,r)}(\mathrm{d}z)\,\big(2f^2+r^{2}(a,b)(\nabla f)^{2}\big)
  \\
 &\le
   2r^2(a,b){\mathcal E}(f,f)+ \int_{[a,b]}\lambda^{(T,r)}(\mathrm{d}z)\,f^2.
\end{aligned}
\ee
Here we have applied in the second line that $2xy\le x^2+y^2$, for all $x,y\in\R$, with $x:=|f|$ and $y:=r(a,b)\cdot|\nabla f|$.
Since $f\in{\mathcal D}({\mathcal E})$ is continuous and hence bounded on [a,b],
it follows that $\tilde{f}_{a,b}\in{\mathcal D}({\mathcal E})$.
\end{proof}\sm}

For technical purposes we also introduce for all $\alpha> 0$
the bilinear form
\begin{equation}\label{con.2b}
   \mathcal E_{\alpha}\big(f,g\big)
 :=
   \mathcal E\big(f,g\big)+\alpha\big(f,g\big)_{\nu}
\ee
with domain
\begin{equation}\label{domainalp}
   {\mathcal D}({\mathcal E}_{\alpha})
 :=
   {\mathcal D}(\mathcal E).
\ee

Moreover, we also consider for any given closed subset $A\subseteq T$ the domain
\begin{equation}\label{con.2a}
    \mathcal D_A(\mathcal E_{\alpha})
 :=
    \mathcal D_A(\mathcal E)
 =
    \big\{f\in\mathcal D(\mathcal E):\,f|_A=0\big\}.
\ee
\sm

The main result of this section states that
the {form
$(\mathcal E,\mathcal D_A(\mathcal E))$ is a
{\em Dirichlet form}}, i.e., symmetric, closed and
Markovian (see, for example,
\cite{FukushimaOshimaTakeda1994} for notation
and terminology).

\begin{proposition}[Dirichlet forms]
For any closed {$A\subseteq T$},\label{P:00}
$(\mathcal E,\mathcal D_A(\mathcal E))$
is a Dirichlet form.
\end{proposition}\sm

\begin{proof}
By an analogous argument, as in Example~1.2.1 in
\cite{FukushimaOshimaTakeda1994}, it can be shown that $({\mathcal E},{\mathcal D}_A({\mathcal E}))$ is
well-defined and symmetric.
The following lemma states that the form $({\mathcal E},{\mathcal D}_A(\mathcal E))$ is closed.
\begin{lemma}[Closed form] For any closed $A\subseteq T$,
\label{L:EDstarclosed}
the form
$({\mathcal E},{\mathcal D}_A(\mathcal E))$
is closed, that is, $\mathcal D_A(\mathcal E)$ equipped with
the inner product $\mathcal E_1$ is complete.
\end{lemma}\sm

\begin{proof} Let $(f_n)_{n\in\N}$ be an $\mathcal E_1$-Cauchy
sequence in $\mathcal D_A(\mathcal E)$. Then there exist
$f,g\in L^2(\nu)$ such that $\lim_{n\to\infty}f_n=f$
in $L^2(\nu)$ and
$\lim_{n\to\infty}\nabla f_n=g$ in $L^2(\lambda^{(T,r)})$.
In particular,
along a subsequence $f=\lim_{k\to\infty}f_{n_k}$, $\nu$-almost
surely. By the Cauchy-Schwartz inequality,
\begin{equation}\label{e:nablaa}
\begin{aligned}
   \Big|\int_x^y&\lambda^{(T,r)}(\mathrm{d}z)\,g(z)-f(y)+f(x)\Big|^2
  \\
 &=
   \lim_{k\to\infty}\Big|\int_x^y\lambda^{(T,r)}(\mathrm{d}z)\,g(z)
   -f_{n_k}(y)+f_{n_k}(x)\Big|^2
  \\
 &=
   \lim_{k\to\infty}\Big|\int_x^y\lambda^{(T,r)}(\mathrm{d}z)\,
   \big(g(z)-\nabla f_{n_k}(z)\big)\Big|^2
  \\
 &\leq
   r(x,y)\lim_{k\to\infty}\big(g-\nabla f_{n_k},g-\nabla
 f_{n_k}\big)_{\lambda^{(T,r)}}
 =
  0,
\end{aligned}
\ee
for $\lambda^{(T,r)}$-almost all $x,y\in T$. Hence
$\nabla f=g$, $\lambda^{(T,r)}$-almost
surely. Similarly, by Fatou's Lemma,
along a subsequence {$(f_{n_l})_{l\in\mathbb{N}}$} with $f=\lim_{l\to\infty}f_{n_l}$, $\lambda^{(T,r)}$-almost
surely,
\begin{equation}\label{e:nabla}
\begin{aligned}
   \lim_{n\to\infty}{\mathcal E}(f_n-f,f_n-f)
 &=
   \lim_{n\to\infty}\int\lambda^{(T,r)}(\mathrm{d}z)\,
   \lim_{l\to\infty}(\nabla f_n(z)-\nabla f_{n_l}(z))^2
  \\
 &\le
   \lim_{n\to\infty}\liminf_{l\to\infty}
   \mathcal E(f_n-f_{n_l},f_n-f_{n_l})
 =
   0.
\end{aligned}
\ee
Clearly, $f|_A=0$ and
the assertion follows.
\end{proof}\sm

{The following lemma shows that the form is {\em contractive}. The conclusions are easily verified, so we omit the proof.

\begin{lemma}[Contraction property] If $f\in{\mathcal D}({\mathcal E})$ then for all $\varepsilon>0$, $f^\varepsilon:=(f\wedge\varepsilon)\vee(-\varepsilon)\in{\mathcal D}({\mathcal E})$ and ${\mathcal E}(f^\varepsilon,f^\varepsilon)\le {\mathcal E}(f,f)$.
\label{L:14}
\end{lemma}\sm

Since the form $({\mathcal E},{\mathcal D}({\mathcal E}))$ is closed and has the contraction property, it immediately follows the it is Markovian
(compare, e.g., {Theorem~1.4.1 in \cite{FukushimaOshimaTakeda1994}}).
\begin{cor}[Markovian form]
For any closed {$A\subseteq T$},
the form
$({\mathcal E},{\mathcal D}_A({\mathcal E}))$
is Markovian, that is, for all $\varepsilon>0$
there exists a Lipschitz continuous function
\label{L:EDstarmarkovian}
$\varphi_\varepsilon:\R\to[-\varepsilon,1+\varepsilon]$ with Lipschitz
constant one such that
\begin{itemize}
\item[(i)]
$\varphi_\varepsilon(t)=t$, for all $t\in[0,1]$, and
\item[(ii)] for all $f\in\mathcal D_A(\mathcal E)$,
$\varphi_\varepsilon\circ f\in\mathcal D_A(\mathcal E)$, and
$\mathcal E(\varphi_\varepsilon\circ f,\varphi_\varepsilon\circ f)
\leq
\mathcal E(f,f)$.
\end{itemize}
\end{cor}\sm}



By Lemma~\ref{L:EDstarclosed} and  Corollary~\ref{L:EDstarmarkovian}, the proof of Proposition~\ref{P:00} is complete.
\end{proof}\sm

We conclude this subsection with the following useful fact.
\begin{lemma}[Transience of $({\mathcal E},{\mathcal D}_A({\mathcal E}))$] \label{L:dtrans} Assume that $(T,r)$ is a locally compact $\R$-tree and $\nu$ a Radon measure on
$(T,{\mathcal B}(T))$.
For any closed, non-empty {$A\subseteq T$} the Dirichlet form
$(\mathcal E,\mathcal D_A(\mathcal E))$ is transient, \label{L:01}
that is, there exists a bounded $\nu$-integrable reference function $g$ which is strictly positive,
$\nu$-almost surely, and satisfies
for all $f\in{\mathcal D}_A({\mathcal E})$,
\begin{equation}\label{e:trans}
   \int\mathrm{d}\nu\,|f|\cdot g
 \le
   \sqrt{{\mathcal E}(f,f)}.
\ee
\end{lemma}\sm

\begin{proof} Let $A\subset T$ be a non-empty and closed subset, and $\rho'\in A$.
\begin{equation}
   g  :=
   {1\wedge} \gamma\sum_{n\in\N}\frac{\mathbf{1}_{\bar{B}(\rho',n)\setminus\bar{B}(\rho',n-1)}}{n^2\nu(\bar{B}(\rho',n)\setminus\bar{B}(\rho',n-1))}
\end{equation}
with a normalizing constant $\gamma:=(\sqrt{2}\sum_{n\ge 1}n^{-3/2})^{-1}$.

Obviously, $g$ is positive and
\be{e:g}
   \int\nu(\mathrm{d}x)g(x)\le\gamma\sum_{n\ge 1}n^{-2}<\infty.
\ee

For all $f\in{\mathcal D}_A({\mathcal E})$ and
$x,y\in T$,
\begin{equation}\label{con.3}
\begin{aligned}
   \big|f(y)-f(x)\big|^2
 &=
   \big|\int^y_x\lambda^{(T,r)}(\mathrm{d}z)\,\nabla f(z)\big|^2
  \\
 &\leq
   2\mathcal E(f,f)r(x,y),
\end{aligned}
\ee
by the Cauchy-Schwartz inequality. Since $f(\rho')=0$, (\ref{con.3}) implies
in particular that $(f(y))^2\le 2\mathcal E(f,f)r(\rho',y)$, and therefore
\begin{equation}
\label{e:trans.1}
\begin{aligned}
   \int\mathrm d\nu\,|f|\cdot g
 &\le
   \sqrt{2}\sqrt{{\mathcal E}(f,f)}\int\mathrm d\nu\,\sqrt{r(\rho',\boldsymbol{\cdot})}\cdot g
  \\
 &\le
\sqrt{2\mathcal E(f,f)}\gamma\sum_{n\ge 1}n^{-3/2}=\mathcal E(f,f)^{1/2}.
\end{aligned}
\ee
\end{proof}\sm

\section{Capacity, Green kernel and resistance}
\label{S:capacities}
In this section we recall well-known facts on capacities and the Green kernel which we will use frequently throughout the paper.
In Subsection~\ref{Sub:extended} we give the notion of the extended Dirichlet space and discuss a frequently used example.
In Subsection~\ref{Sub:capacity} we introduce the capacity and in Subsection~\ref{Sub:Green} the Green kernel. In Subsection~\ref{Sub:resistance}
we discuss the relation between resistance and capacities.
Later on in the article we will relate these potential theoretic notions with the corresponding probabilistic properties of the $\nu$-Brownian motion
associated with the Dirichlet form $({\mathcal E},{\mathcal D}({\mathcal E}))$.\sm

\subsection{Extended Dirichlet space}
\label{Sub:extended}
Let $A\subseteq T$ be a closed set. Let
\begin{equation}
\label{extended}
   (\mathcal E,\bar{\mathcal D}_A(\mathcal E))
  :=
    \mbox{ the extended transient Dirichlet space,}
\ee
i.e., $\bar{\mathcal D}_A(\mathcal E)$ is the family of all Borel-measurable functions $f$ on $T$ such that $|f|<\infty$, $\nu$-almost surely, and there exists a ${\mathcal E}$-Cauchy sequence $\{f_n;\,n\in\N\}$ of functions in $\mathcal D_A(\mathcal E)$ such that $\lim_{n\to\infty}f_n=f$, $\nu$-almost surely. By
Theorem 1.5.3 in \cite{FukushimaOshimaTakeda1994} this space can be
identified with the completion of $\mathcal D_A(\mathcal E)$ with
respect to the inner product~$\mathcal E$.\sm

{\begin{remark}[Connection of Kigami's domain with ${\mathcal D}({\mathcal E})$] \rm
{Recall the forms  $({\mathcal E}^{\mathrm{Kigami}},{\mathcal F}^{\mathrm{Kigami}})$ and $({\mathcal E}^{\mathrm{Kigami}},{\mathcal D}({\mathcal E}^{\mathrm{Kigami}}))$
from (\ref{Kigami:form}) through (\ref{Kigami:labelD}), and the forms  $({\mathcal E},{\mathcal F})$ and $({\mathcal E},{\mathcal D}({\mathcal E}))$
from (\ref{e:F}) through (\ref{con.2p}).
In analogy to (\ref{extended}) write $\bar{D}({\mathcal E}^{\mathrm{Kigami}})$ for the extension of Kigami's domain.


We will now show that
\begin{equation}
\label{e:028}
\bar{{\mathcal D}}({\mathcal E})=\bar{\mathcal D}({\mathcal E}^{\mathrm{Kigami}}).
\ee}

Choose $f\in{\mathcal D}({\mathcal E})={\mathcal F}\cap{\mathcal C}_\infty(T)\cap L^2(\nu)$,
 and put for all  $\varepsilon>0$, $f^\varepsilon:=f-(f\vee(-\varepsilon))\wedge\varepsilon$. Since $f\in{\mathcal C}_\infty(T)$, $f^\varepsilon\in{\mathcal C}_0(T)$
 for all $\varepsilon>0$.
Moreover, $f^\varepsilon\in{\mathcal F}$ for all $\varepsilon>0$.
Since $f^\varepsilon\tvarepsilono f$, $\nu$-almost everywhere and in ${\mathcal E}$ by Theorem~1.4.1 in \cite{FukushimaOshimaTakeda1994}, we find that
$f\in{\mathcal D}({\mathcal E}^{\mathrm{Kigami}})$. This implies that $\bar{{\mathcal D}}({\mathcal E})\subseteq\bar{{\mathcal D}}({\mathcal E}^{\mathrm{Kigami}})$.\sm

On the other hand, if $f\in\bar{{\mathcal D}}({\mathcal E}^{\mathrm{Kigami}})$, then we find an ${\mathcal E}$-Cauchy sequence $(f_n)_{n\in\mathbb{N}}$ in ${\mathcal D}({\mathcal E}^{\mathrm{Kigami}})$
such that $f_n\tno f$, $\nu$-almost everywhere. For each $n\in\mathbb{N}$ we can find, however, a sequence $(h^n_{k})_{k\in\mathbb{N}}$ in ${\mathcal F}\cap{\mathcal C}_0(T)$ such that $h^n_{k}\tko f_n$,  $\nu$-almost everywhere and in ${\mathcal E}$. Thus, along a subsequence $(k_n)_{n\in\mathbb{N}}$ with $k_n\tno\infty$, $h^n_{k_n}\tno f$  $\nu$-almost everywhere and in ${\mathcal E}$. Since ${\mathcal F}\cap{\mathcal C}_0(T)\subseteq{\mathcal F}\cap{\mathcal C}_\infty(T)\cap L^2(\nu)$, $f\in\bar{\mathcal D}({\mathcal E})$ and the claim follows.
\label{Rem:05}
\hfill $\qed$
\end{remark}\sm}

The following will be used frequently.
\begin{lemma} Let $(T,r)$ be a locally compact $\R$-tree and $\nu$ a Radon measure on $(T,{\mathcal B}(T))$.
Assume that $(T,r,\nu)$ is such that $\mathbf{1}_T\in\bar{\mathcal D}({\mathcal E})$.
\label{L:08}
Then for all ${a},{b}\in T$ with ${a}\not ={b}$,
the function $h_{{a},{b}}:=\frac{f_{{a},{b}}}{r({a},{b})}$ with $f_{{a},{b}}$ as defined in (\ref{e:h}) belongs to the extended domain.
\end{lemma}\sm
\begin{proof}  Assume that $\mathbf{1}_T \in \bar{\mathcal D}({\mathcal E})$. Then there exists a
${\mathcal E}$-Cauchy
sequence $\{f_n;\,n\in\N\}$ of functions in ${\mathcal D}({\mathcal E})$ such that $\lim_{n\to\infty}f_n=\mathbf{1}_T$, $\nu$-almost surely.
Fix $a,b\in T$ with $a\not=b$.
For each $n\in\N$, put
$g_{n}:=f_{n}\cdot h_{{a},{b}}$.
By definition, $h_{{a},{b}}$ is a bounded function.  By Example~\ref{Exp:01}, $h_{{a},{b}}$ is absolutely continuous with
{$\nabla h_{{a},{b}}
 =\frac{1}{r({a},{b})}\big(\mathbf{1}_{[{a},{a}\wedge {b}]}-\mathbf{1}_{[{a}\wedge {b},{b}]}\big)$.}
{It follows from Lemma~\ref{L:13} that} $g_{n}\in\mathcal D(\mathcal E)$, and moreover by (\ref{dirform}),
\begin{equation}
\begin{aligned}
\label{e:cca}
   &{\mathcal E}\big(g_{n}-g_m,g_{n}-g_m\big)
  \\
 &=
   {\mathcal E}\big((f_{n}-f_m)\cdot h_{a,b},(f_{n}-f_m)\cdot h_{a,b}\big)
  \\
 &\le
   {2{\mathcal E}\big(f_{n}-f_m,f_{n}-f_m\big)+\tfrac{1}{r(a,b)^2}\int_{[a,b]}\mathrm{d}\lambda^{(T,r)}\,(f_{n}-f_m)^2}
\end{aligned}
\ee
for all $n\in\N$. Since $\{f_n;\,n\in\N\}$ is ${\mathcal E}$-Cauchy, the first summand on the right hand side of (\ref{e:cca}) goes to zero as $m,n\rightarrow\infty$.

As $(f_{n})_{n\in\mathbb{N}}$ converges $\nu$-almost everywhere, there exists $e\in T$ such that $(f_{n} - f_{m}) (e) \rightarrow 0$ as $n,m  \rightarrow \infty$.  Then the second summand is
\begin{equation}
\begin{aligned}
 &=\tfrac{1}{r(a,b)^2}\int_{[a,b]}\mathrm{d}\lambda^{(T,r)}(x) \,\Big ( (f_{n} - f_{m}) (e) + \int_{e}^{x} \mathrm{d}\lambda^{(T,r)} (z) \nabla(f_{n}-f_m)(z) \Big)^2
 \\
&\leq 2 \tfrac{1}{r(a,b)}\left ( (f_{n} - f_{m} (e) \right )^{2}
 \\
&+  2\tfrac{1}{r(a,b)^2}\int_{[a,b]}\mathrm{d}\lambda^{(T,r)}(x) \, r(e,x) \int_{e}^{x} \mathrm{d}\lambda^{(T,r)} (z) \left (\nabla(f_{n}-f_m)(z)\right)^2
\\
&\leq  c_{1} \left [ \left ( (f_{n} - f_{m}) (e) \right )^{2} + {\mathcal E}\big(f_{n}-f_m,f_{n}-f_m\big) \int_{[a,b]}\mathrm{d}\lambda^{(T,r)}(x) \, r(e,x) \right ]\\
&\leq  c_{2} \left [ {\mathcal E}\big(f_{n}-f_m,f_{n}-f_m\big)   +  \left ( (f_{n} - f_{m}) (e) \right )^{2} \right ],
\end{aligned}
\end{equation}
{for suitable constants $c_1$ and $c_2$,}
and tends to $0$ as $m,n  \rightarrow \infty$.
This shows that $h_{{a},{b}}\in\bar{\mathcal D}({\mathcal E})$.
\end{proof}\sm

\subsection{Capacity}
\label{Sub:capacity}
In this subsection we introduce the notion of capacity as a minimizing problem with respect to the Dirichlet form $({\mathcal E}_\alpha,{\mathcal D}_A({\mathcal E}))$. Furthermore we discuss various characterizations of the minimizers. In particular cases we provide explicit formulae for the minimizer.

For any closed  $A\subseteq T$ and another closed set
{$B\subset T\setminus {A}$}, put
\begin{equation}\label{e:LAB}
   \bar{\mathcal L}_{A,{B}}
 :=
   \big\{f\in\bar{\mathcal D}_A(\mathcal E):\,f|_{B}=1\big\}.
\ee
\sm

\begin{definition}[$\alpha$-capacities] For
$\alpha\ge 0$, let the {\rm $\alpha$-capacity} of any closed set ${B}\subseteq T$ with respect to
some other closed set {$A\subset T\setminus {B}$} be defined as
\begin{equation}\label{cap.1}
   \mathrm{cap}^{\alpha}_{A}({B})
 :=
   \inf\big\{\mathcal E_\alpha(f,f)\,:\,f\in\bar{\mathcal L}_{A,{B}}\big\}.
\ee
\label{Def:02}
If $\alpha=0$, we abbreviate $\mathrm{cap}_{A}({B}):=\mathrm{cap}^0_{A}({B})$.
If $A=\emptyset$, we will denote $\mathrm{cap}^{\alpha}_{A}({B})$ by $\mathrm{cap}^{\alpha}({B})$.
{Moreover, if $B=\{b\}$ is singleton, we will write $\mathrm{cap}_{A}(b)$, $\mathrm{cap}^\alpha_{A}(b)$
and $\mathrm{cap}^\alpha(b)$, and so on.}
\end{definition}\sm

We note that one is not restricted to but we shall be content with the
choice of closed sets only.\sm

\begin{lemma}[Non-empty sets have positive capacity] Let $(T,r)$ be a locally compact $\R$-tree and $\nu$ a Radon measure {on $(T,{\mathcal B}(T))$}. {
For 
any $x\in T\setminus A$, $\mathrm{cap}^1_{A}(\{x\})>0$}.
\label{L:02}
\end{lemma}\sm

\begin{proof} We follow the argument in the proof of Lemma~4 in \cite{Kre95} to
show that singletons have positive capacity. By Theorem~2.2.3 in
\cite{FukushimaOshimaTakeda1994} it is enough to show that
for all $x\in T$ the Dirac measures $\delta_x$ is of finite energy integral, i.e.,
there exists a constant $C_x>0$ such that for all $f\in{\mathcal D}({\mathcal E})\cap {\mathcal C}_0(T)$,
\begin{equation}\label{fei}
   f(x)^2\le C_x\,{\mathcal E}_1(f,f)
\ee
(compare (2.2.1) in \cite{FukushimaOshimaTakeda1994}).

Fix $f\in{\mathcal D}({\mathcal E})\cap {\mathcal C}_0(T)$, $x\in T$. Then by (\ref{con.3}) together with $2ab\le a^2+b^2$ applied with $a:=f(y)$ and $b:=(f(x)-f(y))$,
for all $x,y\in T$,
\begin{equation}
\label{e:dirac}
\begin{aligned}
\tfrac{1}{2} f^2(x)&\le |f(x)-f(y)|^2+f^2(y)\\&\le 2{\mathcal E}(f,f)r(x,y)+f^2(y).
\end{aligned}
\end{equation}

Since $(T,r)$ is locally compact we can find a compact neighborhood, $K=K_x$, of $x$.
Integrating the latter over all $y$ with respect to $\mathbf{1}_{K_x}\cdot\nu$ gives
\begin{equation}
\label{e:dirac1}
\begin{aligned}
  \tfrac{1}{2}f^2(x)\nu(K_x)&\le 2{\mathcal E}(f,f)\int_{K_x}\nu(\mathrm{d}y)\,r(x,y)+(f\cdot \mathbf{1}_{K_x},f)_\nu
 \\
 &\le
   2{\mathcal E}(f,f)\int_{K_x}\nu(\mathrm{d}y)\,r(x,y)+(f,f)_\nu
\end{aligned}
\end{equation}

Hence (\ref{fei}) clearly holds with $C_x:=\frac{2\cdot\max \{2\int_{K_x}\nu(\mathrm{d}y)\,r(x,y);1\}}{\nu(K_x)}$.
\end{proof}
\sm

\begin{proposition}[Capacity between two points] Let $(T,r)$ be a locally compact $\R$-tree and a Radon measure $\nu$ {on $(T,{\mathcal B}(T))$}.  Assume furthermore that $(T,r,\nu)$ is such that $\mathbf{1}_T\in\bar{\mathcal D}({\mathcal E})$.
Then for all ${a},{b}\in T$ with ${a}\not ={b}$,
the function $h_{{a},{b}}:=\frac{f_{{a},{b}}}{r({a},{b})}$ with $f_{{a},{b}}$ as defined in (\ref{e:h}) is the unique minimizer of (\ref{cap.1}).
\label{P:04}
In particular,
\begin{equation}
\label{capxy}
   \mathrm{cap}_{{b}}({a}):=\mathrm{cap}_{\{{b}\}}(\{{a}\})=\big(2r({a},{b})\big)^{-1}.
\ee
\end{proposition}\sm

Before providing a proof of the above proposition, we state well-known characterizations of the solution of the minimizing problem (\ref{cap.1}).
\begin{lemma}[Characterization of minimizers; Capacities] Fix a locally compact $\R$-tree $(T,r)$ and a Radon measure $\nu$ on $(T,{\mathcal B}(T))$. Let $A$ be a closed subset,
${B}\subseteq T\setminus A$  be another {non-empty} closed subset and $\alpha\ge 0$.
\begin{itemize}
\item[(i)]
For a function $h^\ast\in\bar{\mathcal L}_{A,B}$ the
following are equivalent: \label{L:05}
\begin{itemize}
\item[(a)] For all $g\in\bar{\mathcal D}_{A\cup B}({\mathcal E})$,
   $\mathcal E_\alpha(h^\ast,g)=0$.
\item[(b)] For all $h\in\bar{\mathcal L}_{A,B}$,
${\mathcal E}_\alpha(h^\ast,h^\ast)\le{\mathcal E}_\alpha(h,h)$.
\end{itemize}
\item[(ii)] If $\bar{\mathcal L}_{A,B} \neq \emptyset$ then there exists a unique function $h^\ast\in\bar{\mathcal L}_{A,B}$ with $h^\ast$ is $[0,1]$-valued and
$\mathcal{E}_\alpha(h^\ast,h^\ast)=\mathrm{cap}_{A}^\alpha({B})$.
\end{itemize}
\end{lemma}\sm

\begin{proof}  (i) {\bf (b) $\Longrightarrow$ (a).}
 Assume that $h^\ast\in\bar{\mathcal L}_{A,B}$ is such that ${\mathcal E}_\alpha(h^\ast,h^\ast)\le{\mathcal E}_\alpha(h,h)$ for all $h\in\bar{\mathcal L}_{A,B}$.  Choose a function  $g\in\bar{\mathcal D}_{A\cup B}({\mathcal E}_\alpha)$, and put $h^{\pm}=h^\ast\pm\varepsilon g$. Then $h^\pm\in\bar{\mathcal L}_{A,B}$ and
\begin{equation}
\begin{aligned}
   {\mathcal E}_\alpha\big(h^\ast,h^\ast\big)
 &\le
   {\mathcal E}_\alpha\big(h^{\pm},h^{\pm}\big)
  \\
 &=
   {\mathcal E}_\alpha\big(h^\ast,h^\ast\big)+\varepsilon^2{\mathcal E}_\alpha\big(g,g\big)\pm 2\varepsilon {\mathcal E}_\alpha\big(g,h^\ast\big),
\end{aligned}
\ee
or equivalently,
\begin{equation}\label{emin3}
\begin{aligned}
   2\big|{\mathcal E}_\alpha\big(g,h^\ast\big)\big|
 &\le
   \varepsilon{\mathcal E}_\alpha\big(g,g\big).
\end{aligned}
\ee
Letting $\varepsilon\downarrow 0$ implies that ${\mathcal E}_\alpha\big(g,h^\ast\big)=0$, which proves (a) since
$g\in\bar{\mathcal D}_{A\cup B}({\mathcal E}_\alpha)$ was chosen arbitrarily.
\sm


{\bf (a) $\Longrightarrow$ (b).} Assume that (a) holds. Then for each $h\in\bar{\mathcal L}_{A,B}$,
$g_{h}:=h^\ast-h\in\bar{\mathcal D}_{A\cup B}({\mathcal E})$. Therefore
\begin{equation}\label{emin}
\begin{aligned}
   {\mathcal E}_\alpha\big(h,h\big)
 &=
   {\mathcal E}_\alpha\big(h^\ast-g_h,h^\ast-g_h\big)
  \\
 &=
   {\mathcal E}_\alpha\big(h^\ast,h^\ast\big)+{\mathcal E}_\alpha\big(g_h,g_h\big)
  \\
 &\ge
    {\mathcal E}_\alpha\big(h^\ast,h^\ast\big).
\end{aligned}
\ee\sm

 (ii) See Theorem~2.1.5 in \cite{FukushimaOshimaTakeda1994}.
\end{proof}\sm

\begin{proof}[Proof of Proposition~\ref{P:04}] Fix $a,b\in T$ with $a\not=b$. Recall from Lemma~\ref{L:08} that under the assumption $\mathbf{1}_T\in \bar{\mathcal D}({\mathcal E})$, also
$h_{{a},{b}}\in \bar{\mathcal D}({\mathcal E})$.

Since for any $g\in\bar{\mathcal D}_{\{{a},{b}\}}(\mathcal E)$,
\begin{equation}\label{e:dir.6}
\begin{aligned}
  \mathcal E\big(h_{{a},{b}},g\big)
 &=
  \frac{1}{2r({a},{b})}\int\mathrm{d}\lambda^{(T,r)}\,\big(\mathbf{1}_{[{a},{a}\wedge {b}]}-\mathbf{1}_{[{b},{a}\wedge {b}]}\big)\cdot \nabla g
  \\
 &=
  \frac{g({a})-g({a}\wedge {b})-g({b})+g({a}\wedge {b})}{2r({a},{b})}
  \\
 &=0,
\end{aligned}
\ee
by (\ref{e:nablah}),
$h_{{a},{b}}$ is the unique minimizer by Lemma~\ref{L:05}. In particular,
it follows from Lemma~\ref{L:13} that
$\mathrm{cap}_{b}(a)={\mathcal E}(h_{a,b},h_{a,b})=(2r(a,b))^{-1}$.
\end{proof}\sm

\subsection{Green kernel}
\label{Sub:Green}
To prove  the characterization of occupation time measure of the process associated with the Dirichlet form $({\mathcal E},{\mathcal D}({\mathcal E}))$ as stated in Proposition~\ref{P:prop} we introduce a more general variational
problem. Its solution corresponds to the {\em Green kernel}. Consider a closed  subset  $A\subset T$.
Let $\kappa$ be a positive finite measure with $\int\mathrm{d}\kappa\,r({\rho},\boldsymbol{\cdot})<\infty$. For each $\alpha\ge 0$ consider the following
variational problem:
\begin{equation}\label{greenvar}
   H^{\alpha,{\rho},\kappa}_{A}
 :=
   \inf\big\{\mathcal E_\alpha(g,g)-{2}\int{\mathrm{d}\kappa\, g};\,g\in\bar{\mathcal D}_A(\mathcal E_\alpha)\big\}.
\ee\sm

There is a well-known characterization of the unique solution to (\ref{greenvar}). \sm

\begin{lemma}[Characterization of minimizers; Green kernel] Let $(T,r)$ be a locally compact $\R$-tree, {$\nu$ a Radon measure on $(T,{\mathcal B}(T))$,}
$A\subseteq T$ be a  closed subset,
$\kappa$ be a positive and finite measure with $\int\mathrm{d}\kappa\,
r(\rho,\boldsymbol{\cdot})<\infty$, {for some (and therefore all) $\rho\in T$,} and $\alpha\ge 0$.
\begin{itemize}
\item[(i)]
For a function $g^\ast\in\bar{{\mathcal D}}_{A}({\mathcal E}_\alpha)$ the
following are equivalent: \label{L:03}
\begin{itemize}
\item[(a)] For all $g\in\bar{\mathcal D}_{A}({\mathcal E}_\alpha)$,
$\mathcal E_\alpha(g^\ast,g)=\int\mathrm{d}\kappa\,g$.
\item[(b)] For all $g\in\bar{{\mathcal D}}_{A}({\mathcal E}_\alpha)$,
${\mathcal E}_\alpha(g^\ast,g^\ast)-{2}\int\mathrm{d}\kappa\,g^\ast\le
{\mathcal E}_\alpha(g,g)-{2}\int\mathrm{d}\kappa\,g$.
\end{itemize}
\item[(ii)] Assume $\bar{{\mathcal D}}_{A}({\mathcal E}_\alpha) \not = \emptyset$. There exists a unique minimizer $g^\ast\in \bar{{\mathcal D}}_{A}({\mathcal E}_\alpha)$ for (\ref{greenvar}).
\item[(iii)] If  $g^{\ast,\alpha,\kappa}\in\bar{{\mathcal D}}_{A}({\mathcal E}_\alpha)$
is {the} minimizer for (\ref{greenvar}) then $g^{\ast,\alpha,\kappa}$ is non-negative.
\end{itemize}
\end{lemma}\sm

\begin{proof}
(i) Proof is very similar to that of Lemma \ref{L:05} (i). So we omit it here.\sm

(ii)  Assume that if $(f_n)_{n\in\N}$  is a minimizing sequence { in $\bar{\mathcal D}_A(\mathcal E_\alpha)$, i.e.,
\begin{equation}
\label{e:minim}
   {\mathcal E}_\alpha(f_n,f_n)-{2}\int\mathrm{d}\kappa\,f_n\tno H_A^{\alpha,{\rho},\kappa}.
\ee}

Notice first that
for all $f\in{\mathcal D}_A(\mathcal E_{\alpha})$,
\begin{equation}\label{e:squareenerint}
\begin{aligned}
   \big(\int_{{T}}\mathrm d\kappa\,|f|\big)^2
 &\leq
   \kappa(T)\cdot\int_{{T}}\mathrm{d}\kappa\,f^2
  \\
 &\leq
   2\kappa(T)\int_{{T}}\mathrm{d}\kappa\,r(\rho,\boldsymbol{\cdot})\cdot \mathcal E(f,f),
\end{aligned}
\ee
where we have applied (\ref{con.3}) with $y:={\rho}$ and used that $f(\rho)=0$.
The latter implies that, in particular, $(\int\mathrm{d}\kappa\, f_n)_{n\in\N}$ is bounded.

Hence, for all $n,l\in\mathbb{N}$,
{\begin{equation}
\begin{aligned}
    &\big (
    \mathcal E_\alpha\big(\tfrac{f_n-f_{n+l}}{2},\tfrac{f_n-f_{n+l}}{2}\big)-{2}\int\mathrm{d}\kappa\,\tfrac{f_n-f_{n+l}}{2}\big )+
    H^{\alpha,{\rho},\kappa}_{A}
  \\
 &\le
    \big (
    \mathcal E_\alpha\big(\tfrac{f_n-f_{n+l}}{2},\tfrac{f_n-f_{n+l}}{2}\big)-{2}\int\mathrm{d}\kappa\,\tfrac{f_n-f_{n+l}}{2}\big)
  \\
 &\qquad+
    \big(
    \mathcal E_\alpha\big(\tfrac{f_n+f_{n+l}}{2},\tfrac{f_n+f_{n+l}}{2}\big)-{2}\int\mathrm{d}\kappa\,\tfrac{f_n+f_{n+l}}{2}\big)
  \\
 &=
   \mathcal E_\alpha\big(\tfrac{f_n}{2},\tfrac{f_n}{2}\big)+\mathcal E_\alpha\big(\tfrac{f_{n+l}}{2},\tfrac{f_{n+l}}{2}\big)-{2}\int\mathrm{d}\kappa\,\tfrac{f_n}{2}-{2}\int\mathrm{d}\kappa\,\tfrac{f_{n+l}}{2}
   -{2}\int\mathrm{d}\kappa\,\tfrac{f_n-f_{n+l}}{2}.
 \end{aligned}
\ee}

It follows from (\ref{e:minim}) that
\begin{equation}\label{e:dir.4al}
\begin{aligned}
    \limsup_{n\to\infty}\sup_{l\in\mathbb{N}}\mathcal E_\alpha\big(f_n-f_{n+l},f_n-f_{n+l}\big)
 &=0,
 \end{aligned}
\ee
i.e.,\ and $(f_n)_{n\in\N}$ is proven to be ${\mathcal E}_1$-Cauchy. By completeness, a limit $f\in\bar{\mathcal D}_A({\mathcal E}_\alpha)$ exists.\sm

Uniqueness, follows easily by an application of Riesz representation Theorem (see Theorem 13.9 \cite{AliBor1999}).\sm

(iii) Since the form $({\mathcal E}_\alpha,{\mathcal D}_A({\mathcal E}))$ is Markovian,
$(0\vee h)\in {\mathcal D}_A({\mathcal E})$ whenever
$h\in{\mathcal D}_A({\mathcal E})$
(See, Theorem~1.4.2 in \cite{FukushimaOshimaTakeda1994}).
Furthermore,
\begin{equation}
\label{e:gwedge0}
\begin{aligned}
   &{\mathcal E}_\alpha(0{\vee} g^{\ast,\alpha,\kappa},0{\vee} g^{\ast,\alpha,\kappa})-{2}\int\mathrm{d}\kappa\,0{\vee} g^{\ast,\alpha,\kappa}
  \\
 &\le
   {{\mathcal E}_\alpha(g^{\ast,\alpha,\kappa},g^{\ast,\alpha,\kappa})-{2}\int\mathrm{d}\kappa\,g^{\ast,\alpha,\kappa}}
\end{aligned}
\ee
where equality holds if $0{\vee} g^{\ast,\alpha,\kappa}=g^{\ast,\alpha,\kappa}$, $\kappa$-, $\nu$-almost surely.
This however implies that $0{\vee} g^{\ast,\alpha,\kappa}=g^{\ast,\alpha,\kappa}$, which proves the claim.
\end{proof}\sm

Consequently, we arrive at the following definition.
\begin{definition}[Green kernel] Let $(T,r)$ be a locally compact $\R$-tree,
$A\subseteq T$ a
closed subset,
$\kappa$ a positive and finite measure with $\int\mathrm{d}\kappa\,r({\rho},\boldsymbol{\cdot})<\infty$,
and $\alpha\ge 0$.
A {\em Green kernel}  $g_A^\alpha\big(\kappa,\boldsymbol{\cdot}\big)$ is the minimizer
for (\ref{greenvar}).
\label{Def:03}
For $x\in T$, we use the abbreviations $g^{\ast,\alpha}_A(x,\boldsymbol{\cdot}):=
g^{\ast,\alpha}_A(\delta_x,\boldsymbol{\cdot})$ and
$g^{\ast,\alpha}_x(\kappa,\boldsymbol{\cdot}):=g^{\ast,\alpha}_{\{x\}}(\kappa,\boldsymbol{\cdot})$.
{For $A:=\emptyset$, we simply write $g^{\ast,\alpha}(x,\boldsymbol{\cdot})$ and $g^{\ast,\alpha}(\kappa,\boldsymbol{\cdot})$, respectively.}
\end{definition}\sm

We conclude this section with providing an explicit formula for the Green kernel in some specific cases.
\begin{proposition}[Green kernel; an explicit formula]
{Let $(T,r)$ be a locally compact $\R$-tree and $\nu$ a {\em Radon measure} on $(T,{\mathcal B}(T))$.}
Fix $A\subseteq T$ non-empty and closed.    Let  $\kappa$ a positive and finite measure with $\int\mathrm{d}\kappa\,r({\rho},\boldsymbol{\cdot})<\infty$, for some (and therefore all) $\rho\in A$, and $\alpha\ge 0$. Assume further  that $h^{\ast,\alpha}_{A,\boldsymbol{\cdot}}$, the unique minimizer to (\ref{cap.1}), exists. The Green kernel is given by\label{P:06}
\begin{equation}\label{GBmin}
   g_A^{\ast,\alpha}\big(\kappa,\boldsymbol{\cdot}\big)
 :=
   \int\kappa(\mathrm{d}x)\,\frac{h^{\ast,\alpha}_{A,\boldsymbol{\cdot}}(x)}
   {\mathrm{cap}^\alpha_{A}(\boldsymbol{\cdot})}.
\end{equation}
\end{proposition}\sm

\begin{proof} Fix $x,y\in T$ with $y\not\in A$. Since
{$h_{A,y}^{\ast,\alpha}\in\bar{\mathcal L}_{A,\{y\}}(\mathcal E_\alpha)$ and ${g_A^{\ast,\alpha}(x,\boldsymbol{\cdot})}\in\bar{\mathcal D}_{A}(\mathcal E_\alpha)$},
\begin{equation}
\label{e:025}
   g_A^{\ast,\alpha}(x,\boldsymbol{\cdot})-g_A^{\ast,\alpha}(x,y)\cdot h_{A,y}^{\ast,\alpha}\in
\bar{\mathcal D}_{\{y\}\cup A}(\mathcal E_\alpha).
\ee

Furthermore, by Lemmata~\ref{L:05} and~\ref{L:03} we find that
\begin{equation}
\label{e:core.1}
\begin{aligned}
   &h^{\ast,\alpha}_{A,y}(x)
  \\
 &=
   {\mathcal E}_\alpha\big(h^{\ast,\alpha}_{A,y}(\boldsymbol{\cdot}),g_A^{\ast,\alpha}(x,{\boldsymbol{\cdot}})\big)
  \\
 &=
   {\mathcal E}_\alpha\big(h^{\ast,\alpha}_{A,y},g_A^{\ast,\alpha}(x,{\boldsymbol{\cdot}})-g_A^{\ast,\alpha}(x,y)\cdot h^{\ast,\alpha}_{A,y}\big)
   +{\mathcal E}_\alpha\big(h^{\ast,\alpha}_{A,y},g_A^{\ast,\alpha}(x,y)\cdot h^{\ast,\alpha}_{A,y}\big)
  \\
 &=
   g_A^{\ast,\alpha}(x,y)\cdot{\mathcal E}_\alpha\big(h^{\ast,\alpha}_{A,y},h^{\ast,\alpha}_{A,y}\big)
  \\
 &=
   g_A^{\ast,\alpha}(x,y)\cdot\mathrm{cap}^\alpha_A(y),
\end{aligned}
\ee
which implies (\ref{GBmin}).
\end{proof}\sm

\begin{cor}[Green kernel; $\alpha=0$, two points]
 Let $(T,r)$ be a locally compact $\R$-tree
 and $\nu$ a Radon measure  on $(T,{\mathcal B}(T))$. Assume furthermore that $(T,r,\nu)$ is such that
 $\mathbf{1}_{T}\in\bar{\mathcal D}({\mathcal E})$. Then for all
 $x,y\in T$ with $x\not =y$, the Green kernel is given by \label{Cor:01}
\begin{equation}\label{GBmin2}
   g^\ast_y\big(x,\boldsymbol{\cdot}\big)
 :=
   2\cdot r\big(c(\boldsymbol{\cdot},x,y),y\big).
\ee
\end{cor}\sm

\begin{proof} {Fix $y,z\in T$, and let $h_{y,z}$ be as defined in Lemma~\ref{L:08}. Since $h_{y,z}\in\bar{\mathcal L}_{y,z}$
by assumption of the corollary together with
Lemma~\ref{L:08}, we can follow from part(ii) of Lemma~\ref{L:05} that a unique minimizer $h^\ast_{y,z}$ to (\ref{cap.1}) exists.}   So we {are in a position to} apply Proposition~\ref{P:06} with $A:=\{y\}$, $\alpha:=0$ and $\kappa:=\delta_{x}$. {Thus,} $g^\ast_y(x,z)=\tfrac{h^\ast_{y,z}(x)}{\mathrm{cap}_{y}(z)}$.
By Proposition~\ref{P:04}, $h^\ast_{y,z}(x)=\frac{r(c(z,x,y),y)}{r(z,y)}$ and $\mathrm{cap}_{y}(z)=\tfrac{1}{2r(z,y)}$. The result {therefore} follows immediately.
\end{proof}\sm

\begin{remark}[Resolvent]\rm \label{resolvent} For $x,y \in T$ and a bounded measurable $f: T \rightarrow \R$, put
\begin{equation}
\label{e:026}
   G^{y}f(x)
 :=
   \int_T\mathrm{d}\nu\, {g^\ast_y\big(\boldsymbol{\cdot},x\big)}\cdot f.
\ee

By Lemma~\ref{L:03}(i),
\begin{equation}
\label{e:027}
   {\mathcal E}\big(G^{y}f, h\big)
 :=
   \int_T\mathrm{d}\nu\, h\cdot f,
\ee
for all $h\in\bar{\mathcal D}_{y}({\mathcal E})$. As usual, we refer to $G^{y}$ as the resolvent corresponding to ${\mathcal E}$.
{\hfill$\qed$}
\end{remark}\sm

\subsection{Relation between resistance and capacity}
\label{Sub:resistance}
In this subsection we define a notion of resistance and discuss its connection to capacity. We will use this in Section~\ref{S:transinfty} where we provide the proof of Theorem~\ref{T:trareha}.

Fix a root $\rho\in T$, assume  $E_{\infty} \not = \emptyset$, and recall from (\ref{brach}) the last common lower bound $x\wedge y$ for any two $x,y\in  E_{\infty}$. We define the {\em mutual energy},
$\bar{\mathcal E}_\rho(\pi,\mu)$, of two probability measures $\pi$ and $\mu$ on $(E_\infty,\mathcal{B}(E_\infty))$  by
\begin{equation}
\label{e:014b}
   \bar{\mathcal E}_\rho\big(\pi,\mu\big)
 :=
   2\int\pi(\mathrm{d}x)\,\int\mu(\mathrm{d}y)\,r\big(\rho,x\wedge y\big).
\ee

Moreover, we introduce
the corresponding {\em resistance} of $T$ {with respect to $\rho$} by
\begin{equation}
\label{e:resis}
   \bar{\mathrm{res}}_{\rho}
 :=
   \inf\big\{\bar{\mathcal E}_\rho(\pi,\pi):\,\pi\in \mathcal
   M_1(E_\infty)\big\},
\ee
where ${\mathcal M}_1(E_\infty)$ denotes the space of all probability measure on $(E_\infty,{\mathcal B}(E_\infty))$.
\sm


\begin{proposition}
Let $(T,r)$ be a locally compact and unbounded $\R$-tree and $\nu$ be a Radon measure on $(T,{\mathcal B}(T))$.  $\rho\in T$ a distinguished root. Then for all $\rho\in T$, \label{P:07}
\begin{equation}\label{P:07.1a}
   \bar{\mathrm{res}}_{\rho}
 \ge
   {\big(\mathrm{cap}(\rho)\big)^{-1}}.
\ee
\end{proposition}\sm

The proof of Proposition~\ref{P:07} relies on the following lemma.
\begin{lemma}
Let $(T,r)$ be a locally compact $\R$-tree
such that $E_\infty \not = \emptyset$, { and $\nu$ be a Radon measure on $(T,{\mathcal B}(T))$.}
\label{L:07}
{For all $\pi\in{\mathcal M}_1(E_\infty)$ and  $h\in\bar{\mathcal D}({\mathcal E})$ with $h(\rho)=1$, }
\begin{equation}\label{e:020}
  \bar{\mathcal E}_\rho\big(\pi,\pi\big)\cdot {\mathcal E}\big(h,h\big)\ge 1.
\ee
\end{lemma}\sm

\begin{proof}  We follow an idea of \cite{Lyo90}.
Notice first that by Fubini's theorem,
\begin{equation}\label{e:022}
\begin{aligned}
   \bar{\mathcal E}_\rho(\pi,\pi)
 &=
   2\int\pi(\mathrm{d}x)\int\pi(\mathrm{d}y)\,
   \int_{[\rho,x\wedge y]}\lambda^{(T,r)}(\mathrm{d}z)
  \\
 &=
   2\int\lambda^{(T,r)}(\mathrm{d}z)\,\pi\big\{x\in E_\infty:\,{z\in x(\R_+)}\big\}^2.
\end{aligned}
\ee

By the Cauchy-Schwarz inequality,
\begin{equation}\label{e:021}
\begin{aligned}
   {\mathcal E}\big(h,h\big)\bar{\mathcal E}_\rho\big(\pi,\pi\big)
 &\ge
   \Big(\int\lambda^{(T,r)}(\mathrm{d}z)\nabla h(z)
   \pi\{x\in E_\infty:\,{z\in x(\R_+)}\Big)^2
  \\
 &=
   \Big(\int\pi(\mathrm{d}y)\,\int_{{y(\R_+)}}\lambda^{(T,r)}(\mathrm{d}z)\nabla
   h(z)\Big)^2
  \\
 &=
   \big(\int\pi(\mathrm{d}y)\, h(\rho) \big)^2 = 1,
\end{aligned}
\ee
and the claim follows.
\end{proof}\sm

\begin{proof}[Proof of Proposition~\ref{P:07}]
The statement holds trivially when  $ \bar{\mathrm{res}}_{\rho}=\infty$.

Assume therefore that  $\bar{\mathrm{res}}_{\rho}<\infty$. By (\ref{e:020}),
\begin{equation}
\label{e:proofres}
\begin{aligned}
   \bar{\mathrm{res}}_{\rho}
 &=\inf\big\{\bar{{\mathcal E}}_\rho(\pi,\pi):\,\pi\in{\mathcal M}_1(E_\infty)\big\}
  \\
 &\ge
   \big(\inf\{{\mathcal E}(h,h):\,h\in\bar{{\mathcal D}}({\mathcal E}),h(\rho)=1\}\big)^{-1}
  \\
 &=\big(\mathrm{cap}({\rho})\big)^{-1}.
\end{aligned}
\end{equation}
\end{proof}\sm

\section{Existence, uniqueness, basic properties (Proof of Theorem~\ref{T:01})}
\label{S:existence}
In this section we establish existence and uniqueness (up to $\nu$-equivalence) of a strong Markov process associated with the Dirichlet form $({\mathcal E},{\mathcal D}({\mathcal E}))$.
The proof will rely on regularity as specified by the following proposition.
{\begin{proposition}[Regularity]
Let $(T,r)$ be a locally compact $\R$-tree
and $\nu$ a Radon measure on $(T,{\mathcal B}(T))$. 
Then the Dirichlet form $({\mathcal E},{\mathcal D}({\mathcal E}))$ is {\em regular}, i.e.,
\begin{itemize}
\item[(i)] ${\mathcal D}({\mathcal E})\cap {\mathcal C}_0(T)$ is dense in ${\mathcal D}({\mathcal E})$ with respect to the topology generated by ${\mathcal E}_1$.
\item[(ii)] ${\mathcal D}({\mathcal E})\cap {\mathcal C}_0(T)$ is dense in ${\mathcal C}_0(T)$ with respect to the uniform topology.
\end{itemize}
\label{L:04}
\end{proposition}\sm

The proof will rely on the following lemma:

\begin{lemma} Let $(T,r)$ be a locally compact and complete $\R$-tree, and $A\subseteq T$ non-empty and closed.
${\mathcal F}\cap {\mathcal C}_0(T)$ is dense in ${\mathcal C}_0(T)$ with respect to the uniform topology.
\label{L:09}
\end{lemma}}\sm

For the proof we shall borrow the ideas from the proof of Lemma~5.13 in~\cite{Kigami95}.  A semi-direct quoting of the above proof might suffice as well but for completeness and to also illustrate the benefit of the explicit limiting form we present the proof in (more) detail.

\begin{proof}
Fix $\rho\in T$, and $f\in {\mathcal C}_0(T)$. Then there exists $R>0$ such that $f\big|_{B^c(\rho,R)}\equiv 0$. For each $n\in N$ choose $\delta_n>0$ such that
$|f(y)-f(x)|<\tfrac{1}{n}$ whenever $x,y\in B(\rho,R+5\delta_n)$ with $r(x,y)<\delta_n.$

Choose for all $n\in\N$ a finite subset $V_{n} \subset T$ with three properties:
\begin{itemize}
\item[(i)] for all three points $x,y,z\in V_n$ the branch point $c(x,y,z)\in {V_n}$ and;
\item[(ii)] $\cup_{z \in V_{n}}  B(z, \frac{\delta_{n}}{2}) \supset \bar{U}_{n}$
\item[(iii)] If $W$ is a connected component of $T\setminus V_n$ with $\mbox{diam}{W}^{(T,r)} > \delta_{n}$ then $W \cap U_{n} =\emptyset$.
\end{itemize}

Denote
\begin{equation}
\label{e:DV}
   D(V_n)
 :=
   \big\{\bar{W}:\,W\mbox{ is a connected component of }T\setminus V_n\big\},
\ee
and let  $\partial{W}:=\bar{W}\cap V$ for all $\bar{W}\in D(V_n)$.
Notice that by the above properties of $V_{n}$ the $\partial{W}$ is either one or two points.

Consider for each $p\in V_n$ the function $h_{p,V_n\setminus\{p\}}$ on $T$ which is the linear interpolation on the subtree, $\mathrm{span}(V_n)$, spanned by $V_n$ with respect to the constrain $h_{p,V_n\setminus\{p\}}\big|_{V_n}=\mathbf{1}_p$ and which satisfies $\nabla h_{p,V_n\setminus\{p\}}\big|_{V^c_n}\equiv 0$. In particular, on each portion of $W$ not in the subtree spanned by $V_{n}$ it is extended as a constant by its value at the appropriate branch point.

Put for all $n\in\mathbb{N}$,
\begin{equation}
\label{e:tildef}
   \tilde{f}_n
 :=
   \sum_{p\in V_n}f(p){h}_{p,V_n\setminus\{p\}}.
\end{equation}
Clearly $\tilde{f}_{n} \in {\mathcal F}.$
Let  $W $ be such that $\mbox{diam}^{(T,r)}(W) \leq \delta_{n}.$
For $x\in W$ and $p \in \partial W$
\begin{equation}
\label{e:ftildef}
\begin{aligned}
   \big|f(x)-\tilde{f}_n(x)\big|
 &\le
   \big|f(x)-f(p)\big|+\big|\tilde{f}_n(p)-\tilde{f}_n(x)\big|
  \\
 &\le
   \big|f(x)-f(p)\big|+\sup_{p'\in\partial W}\big|\tilde{f}_n(p)-\tilde{f}_n(p')\big|\le
   \tfrac{2}{n}.
\end{aligned}
\end{equation}

On the other hand, if $\bar{W}\in D(V_n)$ is such that $\mbox{diam}^{(T,r)}(W) > \delta_{n}$ then ${W}\cap K_{n}=\emptyset$ (see, for example, Lemma~5.12 in \cite{Kigami95}). Therefore
$\tilde{f}_{n} = 0 $ on $W$, and the support of $\tilde{f}_{n}$ is contained in $K_{n}$. Thus
\begin{equation}
\label{e:030}
   \sup_{x \in T} \big|f(x)-\tilde{f}_n(x)\big| \leq \tfrac{2}{n}.
\end{equation}


\end{proof}\sm

{\begin{lemma}[Regularity; compact tree] Let $(T,r)$ be a compact $\R$ tree, and $\nu$ a Radon measure on $(T,{\mathcal B}(T))$. Then Proposition~\ref{L:04} holds, i.e the Dirichlet form $({\mathcal E},{\mathcal D}({\mathcal E}))$ is {\em regular}.
\end{lemma}\sm}

\begin{proof} (i) If $(T,r)$ is compact, then ${\mathcal D}({\mathcal E})\cap{\mathcal C}_0(T)={\mathcal D}({\mathcal E})$ and (i) trivially holds. \sm

(ii) Fix $f\in{\mathcal C}_0(T)={\mathcal C}(T)$. Applying Lemma~\ref{L:09} we can find a sequence $(f_n)_{n\in\mathbb{N}}$ in ${\mathcal F}$ such that $\|f_n-f\|_\infty\tno 0$. Since $(T,r)$ is compact and $\nu$ Radon, $f_n\in L^2(\nu)$ for all $n\in\mathbb{N}$, and thus also the second claim immediately follows.
\end{proof}\sm

For general (not necessarily complete) $\R$-trees we will make use of the follow:

\begin{cor} Let $(T,r)$ be a locally compact $\R$-tree and $\nu$ a Radon measure on $(T,{\mathcal B}(T))$.  If $K\subseteq T$ is a compact subset of $T$ and $U\supset K$ an open subset of $T$ such that $\bar{U}$ is compact, then
there exists a function $\psi^{K,U}\in {{\mathcal D}}({\mathcal E})$ such that $0\le \psi^{K,U}\le 1$, $\psi^{K,U}|_K\equiv 1$ and $\mathrm{supp}(\psi^{A,K,U})\subseteq\bar{U}$.
\label{Cor:06}
\end{cor}\sm

\begin{proof} By assumption,    $(\bar{U},r)$ is a compact $\R$-tree.  The Dirichlet form $({\mathcal E},{\mathcal D}({\mathcal E}))$ is a regular Dirichlet form on $L^{2}(\bar{U},\nu)$.   We can find an open subsets $V_{1} , V_{2}$ of $S$ such that $K\subset V_{1}\subset\bar{V_{1}}\subset V_{2}\subset\bar{V_{2}}\subset U.$  By Theorem~4.4.3 in \cite{FukushimaOshimaTakeda1994} the form  $({\mathcal E},{\mathcal D}_{\bar{U}\setminus{V_{2}}}({\mathcal E}))$ is a regular Dirichlet form on $L^{2}(V_{2},\nu)$.

By Urysohn's lemma a continuous function $f: V_{2}\rightarrow \R$ with $f\big|_K\equiv 1$ and $f\big|_{\bar{U}\setminus V_{1}}\equiv 0$. Since, in particular, the Dirichlet form $({\mathcal E},{\mathcal D}_{\bar{U}\setminus{V_{2}}}({\mathcal E}))$ is regular and $f\in{\mathcal C}_0(V_{2})$, we find a function $g\in{\mathcal D}({\mathcal E})_{\bar{U}\setminus{V_{2}}}$ such that $\|g-f\|_\infty<\tfrac{1}{2}$. Put $\psi^{K,U}:=\min\{1,2g\}$ on ${\bar{V_{2}}}$. Obviously, $g\big|_K\ge \tfrac{1}{2}$ and $g\big|_{S\setminus V_{2}}\equiv 0$, and therefore $\psi^{K,U}$ can be extended to all of $T$ such that $\psi^{K,U} \in {{\mathcal D}}({\mathcal E}),$ $\psi^{K,U}\big|_K\equiv 1$ and $\psi^{K,U}\big|_{T\setminus U}\equiv 0$.
\end{proof}\sm


{\begin{proof}[Proof of Proposition \ref{L:04}] Recall that ${\mathcal D}({\mathcal E})\cap {\mathcal C}_\infty(T)={\mathcal D}({\mathcal E})$. Applying Lemma~1.4.2(i) in  \cite{FukushimaOshimaTakeda1994}, $({\mathcal E}, {\mathcal D}({\mathcal E}))$ is proved to be regular if we show that
\begin{equation} \label{inreg}
{\mathcal D}({\mathcal E}) \mbox{ is dense in } {\mathcal C}_\infty(T) \mbox{ with respect to the uniform topology.}
\end{equation}

Fix therefore $f\in {\mathcal C}_\infty(T)$. For each $n \in \N$, we can then a choose a compact set $K_n$ such that $f\big|_{K_{n}^{c}}\le\tfrac1n$.
Moreover, since $(T,r)$ is locally compact, we can find also an open set $U_n\supset K_n$ such that $\bar{U}_{n}$ is compact. We can then choose
a  $\delta_n>0$ such that $|f(y)-f(x)|<\tfrac{1}{n}$ whenever $r(x,y)<\delta_n$ and $x,y\in U_n$.

We generalize the reasoning and the notation of the proof of Lemma~\ref{L:09}, and choose again for all $n\in\N$ a finite subset $V_{n} \subset T$ satisfying the properties~(i) through~(iii) and consider the corresponding piecewise linear functions $h_{p,V_n\setminus\{p\}}$.

 By Corollary~\ref{Cor:06}, for each $n\in\mathbb{N}$ there exists a $[0,1]$-function $\phi_{n}\in \mathcal{D}({\mathcal E})$ such that $\phi_{n}=1 $ on $K_{n}$ and $\phi_{n}  =0$ on $U_{n}^{c}$. This time we put
\begin{equation}
\label{e:tildefn}
   \tilde{f}_n
 :=
   \sum_{p\in V_n}f(p) \phi_{n}{h}_{p,V_n\setminus\{p\}}.
\end{equation}

By the same reasoning we can show that for all $n\in\N$, $\tilde{f}_n\in{\mathcal D}(\mathcal E)$, and that
$\|\tilde{f}_n-f\|_\infty\le\tfrac{1}{n}$.
\end{proof}\sm}

\begin{proof}[Proof of Theorem~\ref{T:01}]
By Proposition ~\ref{L:04} the form $({\mathcal E}, {\mathcal D}({\mathcal E}))$ is regular. Therefore  by  Theorem~7.2.1 in ~\cite{FukushimaOshimaTakeda1994} there exists a $\nu$-symmetric Hunt process\footnote{For an introduction to Hunt processes, see Section A.2 in \cite{FukushimaOshimaTakeda1994}} $B$ on $(T, {\mathcal B}(T))$ whose Dirichlet form is ${\mathcal E}$.
By Theorem~4.2.7, the process $B$ is unique (i.e., the transition probability function is determined up to an exceptional set).  Also the Dirichlet form
$({\mathcal E},{\mathcal D}({\mathcal E}))$
possesses the {\em local property}, i.e., if
$f,g\in{\mathcal D}({\mathcal E})$ have disjoint compact
support then ${\mathcal E}(f,g)=0$. Hence by Theorem~7.2.2 in~\cite{FukushimaOshimaTakeda1994} the process $B$ has  continuous paths.
{Finally, by Lemma~\ref{L:02} there are no trivial exceptional sets, and
the above therefore imply that $B$ is a continuous $\nu$-symmetric strong Markov process.}
\end{proof}\sm

We conclude this section with providing a proof for Proposition~\ref{P:prop} which
identifies the Brownian motion on the real line as the
$\lambda^{(T,r)}$-Brownian motion on $\mathbb{R}$.
\begin{proof} [Proof of Proposition \ref{P:prop}] Let $(T,r)$ be a locally compact $\R$-tree and $\nu$ a Radon-measure on $(T,{\mathcal B}(T))$. Assume that $(T,r,\nu)$ are such that the $\nu$-Brownian motion on $(T,r)$ is recurrent.

(i) Let $(P_t)_{t\ge 0}$ be the semi-group associated with the process.
By (\ref{e:dir.6}) together with Theorem~2.2.1 in
\cite{FukushimaOshimaTakeda1994}, $f_{a,b}$ and $-f_{a,b}$ are excessive (i.e., $P_t f_{a,b}\ge f_{a,b}$ and $P_t(- f_{a,b})\ge -f_{a,b}$) and hence the process $Y:=(Y_t)_{t\ge 0}$ given by
\be{e:YY}
   Y_t
 :=
   f_{a,b}\big(X_t\big)
\ee
is a bounded non-negative martingale.
Hence by the stopping theorem,
$\mathbf{E}^x[Y_0]=\mathbf{E}^x[Y_{\tau_a\wedge\tau_b}]$, for all $x\in T$. Thus,
\be{stopp}
\begin{aligned}
   f_{a,b}(x)
 &=
   r\big(c(x,a,b),b\big)
  \\
 &=
   f_{a,b}(a)\cdot\mathbf{P}^x\big\{\tau_a<\tau_b\big\}+f_{a,b}(b)\cdot\big(1-\mathbf{P}^x\big\{\tau_a<\tau_b\big\}\big),
\end{aligned}
\ee
and hence since $f_{a,b}(b)=0$,
\be{allgemein}
   \mathbf{P}^x\big\{\tau_a<\tau_b\big\}
 =
   \frac{r\big(c(x,a,b),b\big)}{r(a,b)},
\ee
which proves (\ref{Xhit}). \sm

(ii) As the $\nu$-Brownian motion is recurrent by Lemma~\ref{L:02} and Theorem ~4.6.6(ii) in \cite{FukushimaOshimaTakeda1994} $\mathbf{P}^{x}\{\tau_{b} < \infty\} = 1$.   Therefore by Theorem~4.4.1(ii) in \cite{FukushimaOshimaTakeda1994} ,  $Rf(x) = \mathbf{E}^{x}[ \int_{0}^{\tau_{b}} f(B_{s}) ds]$ is the resolvent of the $\nu$-Brownian motion killed on hitting $b$, i.e.,
\begin{equation}
\label{e:031}
  {\mathcal E}(Rf,h) = \int  \mathrm{d}\nu\, h\cdot f,
\ee
for all $h \in \bar{\mathcal D}_{y}(\mathcal E)$.
Consequently, using the uniqueness of the resolvent (see Theorem~1.4.3 in \cite{FukushimaOshimaTakeda1994}), Remark~\ref{resolvent} and Corollary~\ref{Cor:01}, (\ref{Xocc}) follows.
 \end{proof}\sm

\section{Bounded Trees (Proof of Theorems~\ref{T:04} and~\ref{C:mix})}
\label{S:compact}

In this section we consider bounded $\R$-trees. We start by providing the proof for the basic long-term behavior stated in Theorem~\ref{T:04}.
We then restrict to compact $\R$-trees, or equivalently, to recurrent Brownian motions.
In Subsection~\ref{Sub:eigen} we
provide bounds on the spectral gap. In
Subsection~\ref{Sub:mix} we apply the latter to study mixing times,
and provide the proof of Theorem~\ref{C:mix}.

\begin{proof}[Proof of Theorem~\ref{T:04}]  Let $(T,r)$ be a locally compact and bounded $\R$-tree and $\nu$ a Radon measure on $(T,{\mathcal B}(T))$. We will rely on Theorem~1.6.3 in \cite{FukushimaOshimaTakeda1994} which states that the $\nu$-Brownian motion $B$ on $(T,r)$
is recurrent if and only if $\mathbf{1}_T\in\bar{\mathcal D}({\mathcal E})$ and ${\mathcal E}(\mathbf{1}_T,\mathbf{1}_T)=0$. \sm

Assume first that $(T,r)$ is {\em compact}. In this case, ${\mathcal C}_\infty(T)={\mathcal C}(T)$, and thus $\mathbf{1}_T\in{\mathcal D}({\mathcal E})$. {Clearly}, ${\mathcal E}(\mathbf{1}_T,\mathbf{1}_T)=0$. Hence $B$ is recurrent. Moreover it follows from Proposition~\ref{P:prop}  (with the choice $f\equiv 1$ in (\ref{Xocc}))
that $\mathbf{E}^x[\tau_b]\le 2\nu(T)\cdot r(x,b)<\infty$ for all $b,x\in T$. Hence $\nu$-Brownian motion on compact $\R$-trees is positive recurrent.
\sm

If $(T,r)$ is {\em not compact}, then we can find an $x\in\partial T:=\bar{T}\setminus T$, where $\bar{T}$ here denotes the completion of $T$. Let $x$ be such a ``missing boundary point'' and fix a
Cauchy-sequence $(x_n)_{n\in\N}$ in $(T,r)$ which converges to $x$ in $\bar{T}$. Then for each compact subset $K\subset T$ there are only finitely many points of $(x_n)_{n\in\N}$ covered by $K$. It therefore follows for any
$f\in{\mathcal D}({\mathcal E})\subset{\mathcal C}_{\infty}(T)$ that $\lim_{n\to\infty}f(x_n)=0$.

By the definition of the gradient we have for all $y\in T$ and $n\in\N$,
\begin{equation}
\label{e:eqq}
\begin{aligned}
   f(y)
 &=
   f(x_n)+\int_{x_n}^y\mathrm{d}\lambda^{(T,r)}\,\nabla f
  \\
 &\le
   f(x_n)+\sqrt{2\cdot r(y,x_n)\cdot{\mathcal E}(f,f)}.
\end{aligned}
\end{equation}

Letting $n\to\infty$ implies that for all $y\in T$ and $f \in \bar{{\mathcal D}}({\mathcal E})$,
\begin{equation}
\label{e:003}
   \big(f(y)\big)^2\le 2\mathrm{diam}^{(T,r)}(T) \cdot{\mathcal E}(f,f),
\ee
which implies that $\mathbf{1}_T\not \in\bar{\mathcal D}({\mathcal E})$ and that $B$ is transient.
\end{proof}\sm

\subsection{Principle eigenvalue}
\label{Sub:eigen}
In this subsection we
give estimates on the principal
eigenvalue of the $\nu$-Brownian motion on an locally compact and bounded $\R$-tree $(T,r)$.

For a closed and non-empty subset $A\subseteq T$,
denote by
\begin{equation}\label{specvar}
   \lambda_A(T)
 :=
   \inf\big\{\mathcal E(f,f)\,:\,f\in\bar{\mathcal D}_A(\mathcal E),(f,f)_{\nu}=1\big\}
\ee
the {\em principal eigenvalue} (with respect to $A$).

\begin{lemma}[Estimates on the principal eigenvalue]
Fix a locally compact and bounded $\R$-tree $(T,r)$ and a Radon measure $\nu$  on $(T,{\mathcal B}(T))$.
\label{L:11}
Let $A\subseteq T$  be  closed, non-empty and connected subset. Assume that  $h^\ast_{A,x}$  the unique minimizer of (\ref{cap.1}) with $B:=\{x\}$, and $\alpha=0$ exists. Then
\begin{equation}\label{e:11.0}
   \inf_{x\in T}\frac{\mathrm{cap}_{A}(x)}{(\mathbf 1_T,h^\ast_{A,x})_{\nu}}
 \leq
   \lambda_A(T)
 \leq
   \inf_{x\in T}\frac{\mathrm{cap}_{A}(x)}{(h^\ast_{A,x},h^\ast_{A,x})_{\nu}},
\ee
\end{lemma}\sm

To prepare the proof we provide characterizations of the principle eigenvalue which are
very similar to Lemmata~\ref{L:05} and~\ref{L:03}.

\begin{lemma}[Characterization of minimizers; Principle Eigenvalue]
Let $(T,r)$ be a locally compact and bounded $\R$-tree,
$\nu$ a Radon measure on $(T,{\mathcal B}(T))$, and $A\subset T$ a closed and non-empty subset. \label{L:06}
$\lambda_A(T)$ is well-defined and $\lambda_A(T)$ is positive.
\begin{itemize}
\item[(i)] For all $h^\dagger\in\bar{\mathcal D}_A(\mathcal E)$ with $(h^\dagger,h^\dagger)_{\nu}=1$ the following are equivalent.
\begin{itemize}
\item[(a)] For all $g\in\bar{\mathcal D}_{A}({\mathcal E})$,
$\mathcal E(h^\dagger,g)=\lambda_A(T)(h^\dagger,g)_{\nu}$.
\item[(b)] For all $h\in\bar{{\mathcal D}}_{A}(\mathcal E)$ with $(h,h)_{\nu}=1$,
${\mathcal E}(h^\dagger,h^\dagger)\le{\mathcal E}(h,h)$.
\item[(c)] ${\mathcal E}(h^\dagger,h^\dagger)=\lambda_A(T)(h^\dagger,h^\dagger)_\nu$.
\end{itemize}
\item[(ii)] $\lambda_A(T)$ is positive.
\item[(iii)] Any minimizer of (\ref{specvar}) is sign definite.
\end{itemize}
\end{lemma}\sm

\begin{proof}
(i)  Fix $h^\dagger\in\bar{\mathcal D}({\mathcal E})$ such that $(h^\dagger,h^\dagger)_\nu=1$. \sm

{\bf (b) $\Longrightarrow$ (a).}
Assume that for all $h\in\bar{\mathcal D}_{A}({\mathcal E})\setminus\{h^\dagger\}$ with $(h,h)_\nu=1$, {${\mathcal E}(h^\dagger,h^\dagger)\le{\mathcal E}(h,h)$}.
Fix
$g\in\bar{\mathcal D}_A({\mathcal E})$, and put $h^{\pm}:=h^\dagger\pm\varepsilon (g-(h^\dagger,g)_{\nu}\cdot h^\dagger)$.
Then $h^\pm\in\bar{\mathcal D}_A({\mathcal E})$, and
{\begin{equation}\label{emindag}
\begin{aligned}
   &{\mathcal E}\big(h^\dagger,h^\dagger\big)
  \\
 &\le
   {\mathcal E}\big(h^{\pm},h^{\pm}\big)/ (h^\pm,h^\pm)_\nu
  \\
 &=
   {\mathcal E}\big(h^{\pm},h^{\pm}\big)/ \big(1+\varepsilon^2(g,g)^2_\nu-\varepsilon^2(h^\dagger,g)^2_\nu\big).
\end{aligned}
\ee}

Hence
{\begin{equation}\label{emindag0}
\begin{aligned}
 &{\mathcal E}\big(h^\dagger,h^\dagger\big)\big(1+\varepsilon^2(g,g)^2_\nu-\varepsilon^2(h^\dagger,g)^2_\nu\big)
  \\
 &\leq
   {\mathcal E}\big(h^\dagger,h^\dagger\big)+
\varepsilon^2{\mathcal E}\big(g-(h^\dagger,g)_{\nu}\cdot h^\dagger,g-(h^\dagger,g)_{\nu}\cdot h^\dagger\big)
  \\
 &\qquad\pm 2\varepsilon {\mathcal E}\big(g-(h^\dagger,g)_{\nu}\cdot h^\dagger,h^\dagger\big),
\end{aligned}
\ee}
or equivalently,
{\begin{equation}\label{emin3dag}
\begin{aligned}
   &2\big|{\mathcal E}\big(g-(h^\dagger,g)_{\nu}\cdot h^\dagger,h^\dagger\big)\big|
  \\
 &\le
   \varepsilon{\mathcal E}\big(g-(h^\dagger,g)_{\nu}\cdot h^\dagger,g-(h^\dagger,g)_{\nu}\cdot h^\dagger\big)\\&-\varepsilon(g,g)^2_\nu{\mathcal E}(h^\dagger,h^\dagger)+\varepsilon{\mathcal E}(h^\dagger,h^\dagger)(h^\dagger,g)^2_\nu.
\end{aligned}
\ee}

Letting $\varepsilon\downarrow 0$ implies that ${\mathcal E}(g,h^\dagger)=\lambda_A(T)\cdot(h^\dagger,g)_{\nu}$, which proves (a) since
$g\in\bar{\mathcal D}_{A}({\mathcal E})$ was chosen arbitrarily.\sm

{\bf (a) $\Longrightarrow$ (c).} Assume (a) holds. Then (c) follows with the particular choice $g:=h^\dagger$. \sm

{\bf (c) $\Longrightarrow$ (b).} This is an immediate consequence of the definition (\ref{specvar}). \sm

(ii) Transience (\ref{e:trans}) implies $\lambda_A(T)>0$ if
$A$ is non-empty.\sm

(iii)
 Let $h^\dagger$ be a minimizer. Let
$S_\pm:=\{x\in T\,|\,\pm h^\dagger(x)>0\}$, and put
$h_\pm^\dagger:=\pm\mathbf 1_{S_\pm} h^\dagger = \pm \frac{h^\dagger \pm \mid h^\dagger\mid}{2}$, i.e., $h^\dagger=h^\dagger_+-h^\dagger_-$.

 To  verify that $h^\dagger$ is sign definite, we proceed by contradiction and
assume to the contrary that  $\nu(S_-)\cdot\nu(S_+)>0$.
In this case we can define
\begin{equation}
\label{e:004}
  \tilde{h}
 :=
  \frac{(h_-^\dagger,h_-^\dagger)_\nu h_+^\dagger+(h^\dagger_+,h^\dagger_+)_\nu h_-^\dagger}{\sqrt{2(h_-^\dagger,h_-^\dagger)_\nu(h^\dagger_+,h^\dagger_+)_\nu}} .
\ee
It is easy to see that $\tilde{h}\in\bar{{\mathcal D}}_A({\mathcal E})$ is orthogonal to $h^\dagger$ and that
$(\tilde{h},\tilde{h})_\nu=1$.

Orthogonality together with (a) applied on $g:=\tilde{h}$ implies ${\mathcal E}(h^\dagger,\tilde{h})=0$, while we can also read off from (a) that
\begin{equation}
\label{e:readoff}
\begin{aligned}
0=   &\sqrt{2(h_-^\dagger,h_-^\dagger)_\nu(h^\dagger_+,h^\dagger_+)_\nu}{\mathcal E}(h^\dagger,\tilde{h})
 \\
 &=
   {\mathcal E}(h^\dagger_+-h^\dagger_{-},(h_-^\dagger,h_-^\dagger)_\nu h_+^\dagger+(h^\dagger_+,h^\dagger_+)_\nu h_-^\dagger)
  \\
 &=
   2\lambda_A(T)(h_-^\dagger,h_-^\dagger)_\nu(h^\dagger_+,h^\dagger_+)_\nu.
\end{aligned}
\ee
This, of course, is a contradiction since $\lambda_A(T)>0$.
\end{proof}\sm

\begin{proof}[Proof of Lemma~\ref{L:11}] Let $\varphi_A$ be
a non-negative minimizer of (\ref{specvar}) and  $g^\ast_A(\nu,\boldsymbol{\cdot})$
the unique minimizer of (\ref{greenvar}) with $\kappa:=\nu$ and $\alpha:=0$.
Then by Lemma~\ref{L:06} together with Lemma~\ref{L:07},
\begin{equation}\label{e:11.1}
\begin{aligned}
   \lambda_A(T)
 &=
   \frac{{\mathcal E}(\varphi_A,g^\ast_A(\nu,,\boldsymbol{\cdot}))}{(\varphi_A,g_A^\ast(\nu,\boldsymbol{\cdot}))_{\nu}}
 = \frac{(\varphi_A,\mathbf{1}_T)_{\nu}}{(\varphi_A,g^\ast_A(\nu,\boldsymbol{\cdot}))_{\nu}}
 \geq
   \inf_{x\in T}\big(g_A^\ast(\nu,x)\big)^{-1}.
\end{aligned}
\ee

Moreover, by Proposition~\ref{P:06},
\begin{equation}\label{e:11.4}
   g_A^\ast(\nu,x)
 =
   \frac{(h^\ast_{A,x},\mathbf{1}_T)_{\nu}}{\mathrm{cap}_{A}(x)},
\ee
where $h^\ast_{A,x}$ is the unique minimizer of (\ref{cap.1}) with $\alpha:=0$. This together with (\ref{e:11.1})
implies the {\em lower bound} in (\ref{e:11.0}).

To obtain the {\em upper bound} insert $f_{A,x}:=h^\ast_{A,x}/(h^\ast_{A,x},h^\ast_{A,x})_{\nu}^{1/2}$, $x\not\in A$,
into (\ref{specvar}). Then for all $x\in T\setminus\{A\}$,
\begin{equation}\label{e:11.2}
   \lambda_A(T)
 \leq
   {\mathcal E}(f_{A,x},f_{A,x})
 =
   \frac{{\mathrm{cap}_{A}(x)}}{(h^\ast_{A,x},h^\ast_{A,x})_{\nu}},
\ee
which gives the claimed upper bound.
\end{proof}\sm

The following proposition is an immediate consequence for compact $\R$-trees. The lower bound  in (\ref{P:08.1}) {has been verified for $\alpha$-stable trees} in the proof of Lemma 2.1 of \cite{CroydonHumbly2010}.

\begin{proposition} Fix a compact $\R$ tree and a Radon measure $\nu$ on $(T,{\mathcal B}(T))$.
For all $b\in T$,
\label{P:08}
{\begin{equation} \label{P:08.1}
    \tfrac{1}{2\left (\mathrm{diam}^{(T,r)}(T)\cdot\nu(T)\right)}
 \le
    \lambda_{\{b\}}(T)
  \le
   \tfrac{1}{2}\inf_{x\in T\setminus\{b\}}\big(\nu\big\{y\in T:\,x\in[y,b]\big\}\cdot r(x,b)\big)^{-1}.
\ee}
\end{proposition}\sm

\begin{proof} When $A= \{b\}$, the minimizer to (\ref{cap.1}),  $h^{\star}_{b,x}$ exists.  So the assumptions of Lemma \ref{L:11} and Corollary~\ref{Cor:01} are satisfied.
For the {\em lower bound}, recall from (\ref{e:11.1}) together with Corollary~\ref{Cor:01} that for all $x\in T$ with $x\not=b$,
%
\begin{equation}\label{newe}
\begin{aligned}
   \lambda_{\{b\}}(T)
 &\ge
   \inf_{x\in T\setminus\{b\}}\big(g^\ast_{\{b\}}(\nu,x)\big)^{-1}
  \\
 &=
   \inf_{x\in T\setminus\{b\}}\big(2\int\nu(\mathrm{d}y)\,r(c(x,y,b),b)\big)^{-1}
  \\
 &\ge
   \big(2\cdot\mathrm{diam}^{(T,r)}(T)\cdot\nu(T)\big)^{-1},
\end{aligned}
\ee
as claimed. \sm

For the {\em upper bound}, recall from (\ref{e:11.2}) together with Proposition~\ref{P:04} that
\begin{equation}\label{upper}
\begin{aligned}
  \lambda_b(T)
 &\le
   \inf_{x\in T\setminus\{b\}}\frac{{\mathrm{cap}_{b}(x)}}{(h^\ast_{x,b},h^\ast_{x,b})_{\nu}}
  \\
 &=
   \inf_{x\in T\setminus\{b\}}\frac{r(x,b)}{2\cdot (r(c(\boldsymbol{\cdot},x,b),b),r(c(\boldsymbol{\cdot},x,b),b))_{\nu}}
  \\
 &\le
   \tfrac{1}{2}\inf_{x\in T\setminus\{b\}}\big(\nu\big\{y\in T:\,x\in[y,b]\big\}\cdot r(x,b)\big)^{-1},
\end{aligned}
\ee
where we have used that for all $x\in T\setminus\{b\}$,
\begin{equation}\label{P:08.4}
\begin{aligned}
   \int\nu(\mathrm{d}y)\,r\big(c(y,x,b),b\big)^2
  &\ge
   \nu\big\{y\in T:\,x\in[y,b]\big\}\cdot r(x,b)^2.
\end{aligned}
\ee
\end{proof}\sm

\subsection{Mixing times}
\label{Sub:mix}
In this subsection we give the proof of Theorem~\ref{C:mix}
based on estimates of the spectral gap of the process associated with the Dirichlet form.

Denote by
\begin{equation}
\label{specgapvar}
\begin{aligned}
   \lambda_2(T)
 &:=
   \inf\big\{\mathcal E(f,f)\,:\,f\in\bar{\mathcal D}(\mathcal E),(f,f)_{\nu}=1,(f,\mathbf{1}_T)_{\nu}=0\big\}
\end{aligned}
\ee
{\em the spectral gap}.

Here is a useful characterization of the spectral gap.
\begin{lemma}[Characterization of minimizers; Spectral gap]
Let $(T,r)$ be a compact $\R$-tree and $\nu$ a Radon measure  on $(T,{\mathcal B}(T))$. \label{L:12}
\begin{itemize}
\item[(i)]
For all $h^\ddagger\in\bar{\mathcal D}(\mathcal E)$ with $(h^\ddagger,h^\ddagger)_{\nu}=1$
and $(h^\ddagger,\mathbf{1}_T)_{\nu}=0$ the following are equivalent.
\begin{itemize}
\item[(a)] For all $g\in\bar{\mathcal D}(\mathcal E)$ with $(g,\mathbf{1}_T)_{\nu}=0$,
$\mathcal E(h^\ddagger,g)=\lambda_2(T)(h^\ddagger,g)_{\nu}$.
\item[(b)] For all $h\in\bar{{\mathcal D}}(\mathcal E)$ with $(h,h)_{\nu}=1$
and $(h,\mathbf{1}_T)_{\nu}=0$,
${\mathcal E}(h^\ddagger,h^\ddagger)\le{\mathcal E}(h,h)$.
\item[(c)] ${\mathcal E}(h^\ddagger,h^\ddagger)=\lambda_2(T)$.
\end{itemize}
\item[(ii)] If $h^\ddagger$ is a minimizer to the minimum problem (\ref{specgapvar}), then $\lambda_2(T)\ge\lambda_b(T)$ for all $b\in T$ with $h^\ddagger(b)=0$.
\end{itemize}
\end{lemma}\sm

\begin{proof}
(i) The proof is very similar to that of Lemma \ref{L:06}. We do not repeat it here.

(ii) Fix $h^\ddagger\in\bar{\mathcal D}({\mathcal E})$ such that $(h^\dagger,h^\dagger)_\nu=1$ and $(h^\ddagger,\mathbf{1}_T)_\nu=0$.
Let $h^\ddagger\in\bar{\mathcal D}(\mathcal E)$ be a minimizer corresponding to (\ref{specgapvar}).
Since $(h^\ddagger,\mathbf{1}_T)_\nu=0$, the zero set $S_0:=\{x\in T:\,h^\ddagger(x)=0\}\not=\emptyset$. Moreover,
if $b\in S_0$ then $h^\ddagger\in\bar{\mathcal D}_b({\mathcal E})$ and therefore by Definition~(\ref{specvar}),
$\lambda_{2}(T)={\mathcal E}(h^\ddagger,h^\ddagger)\ge\lambda_{\{b\}}(T)$.
\end{proof}\sm

We close the section with the proof of Theorem~\ref{C:mix}.

\begin{proof}[Proof of Theorem~\ref{C:mix}] Notice first that since $\nu$-Brownian motion is recurrent on compact $\R$-trees, it is conservative.
Consequently, if $(P_t)_{t\ge 0}$ denote the semi-group  then $P_t(\mathbf{1}_T)=\mathbf{1}_T$ for all $t\ge 0$. Thus  we can conclude by $\nu$-symmetry,
for
all probability measures $\nu'$ on $(T,{\mathcal B}(T))$ such that $\nu'\ll\nu$ with  $\tfrac{\mathrm{d}\nu'}{\mathrm{d}\nu}\in L^1(\nu')$,
\begin{equation}\label{e:mix.2}
\begin{aligned}
   \big\|\nu' P_t-\tfrac{\nu}{\nu(T)}\big\|_{\mathrm{TV}}
 &=
   \big\|(P_t\tfrac{\mathrm{d}\nu'}{\mathrm{d}\nu}\nu(T)-\mathbf 1_T)\tfrac{\nu}{\nu(T)}\big\|_{\mathrm{TV}}
  \\
 &=
   \big\|(P_t(\tfrac{\mathrm{d}\nu'}{\mathrm{d}\nu}\nu(T)-\mathbf 1_T)\tfrac{\nu}{\nu(T)}\big\|_{\mathrm{TV}}.
\end{aligned}
\ee

By Jensen's inequality, the assumption that $(\mathbf 1_T,f)_\nu=1$ and
the spectral theorem applied to $P_t$ (see discussion on page 2 in \cite{wang00} and references there in)
\begin{equation}\label{e:mix.3}
\begin{aligned}
   \big\|\nu' P_t-\tfrac{\nu}{\nu(T)}\big\|_{\mathrm{TV}}
 &\leq
   \Big(\int\tfrac{\mathrm{d}\nu}{\nu(T)}\,\big|P_t(\tfrac{\mathrm{d}\nu'}{\mathrm{d}\nu}\nu(T)-\mathbf 1_T)\big|^2\Big)^{1/2}
  \\
 &\leq
   \mathrm e^{-\lambda_2(T) t}\big(\sqrt{\nu(T)}(\mathbf{1}_T,\tfrac{\mathrm{d}\nu'}{\mathrm{d}\nu})_{\nu'}^{1/2}+1\big).
\end{aligned}
\ee

The assertion now follows from (\ref{P:08.1}) and Lemma \ref{L:12}~(ii).
\end{proof}\sm

\section{Trees with infinite diameter}
\label{S:transinfty}

In this section we consider the $\nu$-Brownian motion on a locally compact and unbounded $\R$-trees $(T,r)$. We shall give the proof of Theorem~\ref{T:trareha} which is based on the following criterion for recurrence and transience
relating the potential theoretic and the dynamic
approach in a transparent way. Recall from (\ref{Einfty}) the set $E_\infty$ of ends at infinity.
The following proposition relates transience with a positive capacity between the root and the ends at ``infinity''.

\begin{proposition} \label{P:03}
Let $(T,r)$ be a locally compact $\R$ tree and $\nu$ a Radon measure on $(T,{\mathcal B}(T))$. Then the following are equivalent.
\begin{itemize}
\item[(a)] The $\nu$-Brownian motion on $(T,r)$ is recurrent.
\item[(b)] $\mathrm{cap}(\rho)=0$.
\end{itemize}
\end{proposition}\sm

{
\begin{proof}  By Theorem~1.6.3 in [FOT94], $\nu$ Brownian motion on $(T,r)$ is recurrent if and only if
there exists a sequence $(h_{{k}})_{k\in\mathbb{N}}$ in ${\mathcal D}({\mathcal E})$ such that $h_{{k}}\to 1$, $\nu$-almost everywhere, and
$\mathcal E(h_{{k}},h_{{k}})\to 0$, as $k\to\infty$.

{ $\mathbf (b) \Longrightarrow (a)$: }
Suppose $\mathrm{cap}(\rho)=0$. Then there exists for each $n\in\mathbb{N}$ a function $h_{n}\in\bar{\mathcal D}(\mathcal E)$ with $h_{n}(\rho) = 1$ and such that $\mathcal E(h_{n},h_{n}) \rightarrow 0$, as $n\to\infty$. By standard $L^{2}$-theory there exists a subsequence $\nabla h_{n_{k}}\rightarrow 0$, $\lambda^{(T,r)}$-almost everywhere, as $k\to\infty$. As for each $k\in\mathbb{N}$ and $x\in T$,
\begin{equation}
\label{l:031}
   h_{n_{k}}(x)
 =
   1+\int_{\rho}^{x}\mathrm{d}\lambda^{T,r}\, \nabla   h_{n_{k}},
\ee
$h_{n_{k}}\to 0$ pointwise, as $k\to\infty$.
Thus, $\nu$-Brownian motion on $(T,r)$ is recurrent. \sm

{$\mathbf  (a) \Longrightarrow (b)$: }  Suppose $\nu$-Brownian motion on $(T,r)$ is recurrent, then we
can choose a sequence $(h_{{k}})_{k\in\mathbb{N}}$ in ${\mathcal D}({\mathcal E})$ such that $h_{{k}}\to 1$, $\nu$-almost everywhere, and
$\mathcal E(h_{{k}},h_{{k}})\to 0$, as $k\to\infty$.

Since $\nu$ is Radon there exists an $a\in B(\rho, 1)$ such that $h_{{k}}(a)\to 1$, as $k\to\infty$.
As
\begin{equation}
\label{l:032}
   h_{{k}}(\rho)
 =
   h_{{k}}(a)-\int_{\rho}^{a}\mathrm{d}\lambda^{T,r}\, \nabla   h_{n_{k}},
\ee
Cauchy-Schartz inequality implies that
\begin{equation}
\label{l:033}
   |h_{{k}}(\rho)-1|
 \leq
   |h_{k}(a) -1|  + 2\mathcal E\big(h_{{k}},h_{{k}}\big).
\ee

Therefore $h_{{k}}(\rho)\to 1$, as $k\to\infty$. Consequently, we can assume without loss of generality that $h_{k}(\rho) >0$, for all $k\in\mathbb{N}$.
Put $f_{k}:=\frac{h_{k}}{h_{k}(\rho)}$. It is easy to verify that $f_{k} \in{\mathcal D}({\mathcal E})$ with $f_{k}(\rho)=1$ and such that
$\mathcal E(f_{{k}},f_{{k}})\to 0$, as $k\to\infty$.
This implies that $\mathrm{cap}(\rho)=0$.
\end{proof}\sm}

We conclude this section with the proof of Theorem~\ref{T:trareha}.

\begin{proof}[Proof of Theorem \ref{T:trareha}] (i) Recall from (\ref{Einfty}) and (\ref{e:barr}) the set $E_\infty$ of ends at infinity equipped with the distance $\bar{r}$, and from (\ref{e:inha}) the $1$-dimensional Hausdorff measure ${\mathcal H}^1$ on $(E_\infty,\bar{r})$.
Assume that $(T,r,\nu)$ is such that
\begin{equation}
\label{e:010}
   {\mathcal H}^1\big(E_\infty,\bar{r}\big)<\infty.
\ee

Then for all $\varepsilon\in(0,1)$,
there exists a disjoint finite covering of $E_\infty$ by sets $E_i\subseteq E_\infty$, $i=1,...,m=m(\varepsilon)$, with $\mathrm{diam}^{(E_\infty,\bar{r})}(E_i)\le\varepsilon$ and furthermore the sequence can be chosen so that
\begin{equation}
\label{finitelength}
   \lim_{\varepsilon\rightarrow 0}\sum_{i=1}^{m(\varepsilon)}\mathrm{diam}^{(E_\infty,\bar{r})}(E_i)<\infty.
\ee

For each such collection a so-called finite {\em cut set} $\{x_n;\,n=1,...,m=m(\varepsilon)\}$ in $T$ is given by letting $x_i:=\min E_i$. Note that
$\mathrm{diam}^{(E_\infty,\bar{r})}(E_i)\le\varepsilon$ if and only if $r(\rho,x_i)\ge \varepsilon^{-1}$.
Let $y_{i} \in T$ be such that $y_{i}\in[\rho,x_{i}]$ and $r(\rho,y_i)=\tfrac{r(\rho,x_i)}{2}$. Put $V:=\{x_{i},y_{i}:\,i=1,2,...,m(\varepsilon)\}$, recall from (\ref{e:DV}), the set $D(V)$ of the closure of the connected components of $T\setminus V$. As before, let for all $\bar{W}\in D(V)$, $\partial \bar{W}:=\bar{W}\cap V$. Let for any $p,q\in T$ with $p\not=q$, $h^\ast_{p,q}$ be the minimizer of (\ref{cap.1}) with $\alpha:=0$, $A:=\{q\}$ and $B:=\{p\}$.
Let for each $\varepsilon>0$,
\begin{equation}
\label{e:V1}
  h_{\varepsilon}(x)
 :=
  \sum_{i=1}^{m(\varepsilon)}\mathbf{1}_{\bar{W}_{x_i,y_i}}\cdot h^{\ast}_{y_i,x_i}+\mathbf{1}_{T\setminus\cup_{i=1}^{m(\varepsilon)}(\bar{W}_{x_i,y_i}\cup E_i)}
\end{equation}

Since $\{E_{i};\,i=1,...,m\}$ cover $E_{\infty}$,  the  support of $h_{\varepsilon}$ is a compact set, and therefore in particular, $h_{\varepsilon}\in{\mathcal D}({\mathcal E})$.  Furthermore,
\begin{equation}
\label{e:intnabla1}
\begin{aligned}
 {\mathcal E}\big(h_\varepsilon,h_\varepsilon\big)
 &=
   \tfrac{1}{2}\int\mathrm{d}\lambda^{(T,r)}(\nabla h_\varepsilon)^2
  \\
 &=
    \tfrac{1}{2}\sum\nolimits_{i=1}^{m(\varepsilon)}\big(r(y_{i},x_i)\big)^{-1}
  \\
 &=
    \sum\nolimits_{i=1}^{m(\varepsilon)}\big(r(\rho,x_i)\big)^{-1}
  \\
 &=
    \sum\nolimits_{i=1}^{m(\varepsilon)}\mathrm{diam}^{(E_\infty,\bar{r})}(E_i).
\end{aligned}
\end{equation}

In particular, $h_\varepsilon\in{\mathcal D}({\mathcal E})$ and  $\limsup_{\varepsilon\to 0}\int d\lambda^{(T,r)}(\nabla h_\varepsilon)^2<\infty$ by (\ref{finitelength}).
Moreover, $h_\varepsilon\to\mathbf{1}_T$, as $\varepsilon\to 0$, and an application of H\"older's inequality will yield that  $h_\varepsilon$ is ${\mathcal E}$-Cauchy.  Therefore $h_\varepsilon\to\mathbf{1}_T$ in ${\mathcal E}_1$ as $\varepsilon\to 0$.
That is, $\mathbf{1}_T\in\bar{\mathcal D}({\mathcal E})$ and therefore the $\nu$-Brownian motion is recurrent.  \sm

(ii) Next assume that $\mathrm{dim}_H(E_\infty,\overline r)>1$. Then by the converse of Frostman's energy theorem (compare, e.g., Theorem~4.13(ii) in \cite{Fal2003}) there exists $\pi \in{\mathcal M}_1(E_\infty)$ with $\bar{\mathcal E}(\pi,\pi)<\infty$, and hence $\mathrm{res}_\rho<\infty$. Thus $\nu$-Brownian motion is transient by Proposition~\ref{P:03} together with Proposition~\ref{P:07}.
 \end{proof}\sm

\section{Connection to the discrete world (Proof of Theorem~\ref{T:trareha})}
\label{Sub:contdisc}

In this section we give the proof of Theorem~\ref{nashwillconverse}. It will be concluded from Theorem~\ref{T:trareha} by considering
the embedded Markov chains.  For that notice that we can associate any weighted discrete tree $(V,\{r_{\{x,y\}};\,x,y\in V\})$
with the following locally compact $\R$-tree:  fix a root $\rho\in V$ and introduce the metric
$r_V(x,y)=\sum_{e\in |x,y|}r_e$, $x,y\in V$, where $|x,y|$ is the set of
edges of the self avoiding path connecting $x$ and $y$. Notice that $(V,r_V)$ is a $0$-hyperbolic space, or equivalently,  $r(v_1,v_2)+r(v_3,v_4)\le\max\{r(v_1,v_3)+r(v_2,v_4);r(v_1,v_4)+r(v_2,v_3)\}$ for all $v_1,v_2,v_3,v_4\in V$.
By Theorem~3.38 in~\cite{Eva} we can find a smallest $\R$-tree $(T,r)$
such that $r(x,y)=r_V(x,y)$ for all $x,y\in V$. The following lemma complements the latter to a  one-to-one correspondence between
rooted weighted discrete tree and rooted $\R$-trees.

\begin{lemma}[Locally compact $\R$-trees induce weighted discrete trees]
Let $(T,r,\rho)$ be locally compact rooted $\R$-tree which is spanned by its ends at infinity.
Then the following holds:
\begin{itemize}
\item[(i)] All $x\in T$ are of finite degree, i.e.,the number of
connected components of $T\setminus\{x\}$ is finite.
\item[(ii)] Any ball contains only finitely many branch points, i.e, points of degree as least $3$.
\end{itemize}
\label{L:061}

In particular, $\lambda^{(T,r)}(B(\rho,n))<\infty$, for all $n\in\N$.
\end{lemma}\sm

\begin{remark}[Locally compact $\R$-trees induce weighted discrete trees]\rm
Given a locally compact rooted $\R$-tree which is spanned by its ends at infinity, let $V$ be the set of branch points in $(T,r)$ and $r_{\{x,y\}}=r_{x,y}$
\label{Rem:02}
for all $x,y\in V$ such that $[x,y]\cap V=\emptyset$. Obviously, $(V,\{r_{\{x,y\}};\,x,y\in V\})$ is a weighted discrete tree.
\hfill$\qed$
\end{remark}\sm

\begin{proof} Recall from  Lemma~5.9 in \cite{Kigami95}
that in a locally compact and complete
metric space all closed balls are compact.

(i) We give an indirect proof and assume to the contrary that $x\in T$ is a point of infinite degree. Then
$T\setminus\{x\}$ decomposes in at least countably many connected components, $T_1,T_2,...$ with only leaves in infinite distance to the root, i.e., $T_n=T_n^o$. We can therefore pick points $\{y_1,y_2,...\}$ with $y_i\in T_i$ and $r(x,y_i)=1$, $i=1,...$. Thus the mutual distances between any two of  $\{y_1,y_2,...\}\subseteq B(\rho,r(\rho,x)+2)$ is $2$. This implies that the closed ball $\bar{B}(\rho,r(\rho,x)+2)$ can not be compact. The latter, however, contradicts the local compactness of $(T,r)$.\sm

(ii) Let $n\in\N$ be arbitrary. Assume that $B(\rho,n)$ contains an infinite sequence of mutually distinct branch points $\{y_1,y_2,...\}$.
Since the closed ball $\bar{B}(\rho,n)$ is compact, we can find a subsequence $(n_k)_{k\in\N}$ and a limit point $y\in\bar{B}(\rho,n)$ such that $y_{n_k}\to y$, as $k\to\infty$.  Fix $\varepsilon\in(0,\frac{n}{2})$. Then there is $K=K(\varepsilon)$ such that $y_{n_k}\in B(y,\varepsilon)$ for all $k\ge K$. Moreover, we can pick for any $k\ge K$ a point $z_{n_k}$ such that $y_{n_k}\in[\rho,z_{n_k}]$, $r(y_{n_k},z_{n_k})=\varepsilon$ and $r(z_{n_k},z_{n_l})\ge 2\varepsilon$ for all $l\not =k\ge K$. This, however, again contradicts the fact that $\bar{B}(\rho,n)$ is compact. Since $n$ was chosen arbitrarily, this implies the claim.\sm

Combining the two facts, we can upper estimate $\lambda^{(T,r)}(B(\rho,n))$ by $n$ times the number of branch points in $\lambda^{(T,r)}(B(\rho,n))$
times their maximal degree times $n$. This finishes the proof.
\end{proof}\sm

It follows immediately that $\lambda^{(T,r)}$-Brownian motion $B:=(B_t)_{t\ge 0}$ is well-defined on locally compact $\R$-trees $(T,r)$
which are spanned by their ends at infinity. Let $(V,\{r_{\{x,y\}};\,x,y\in V\})$ be the corresponding weighted discrete tree.

\begin{lemma}[Embedded Markov chain]
Let $(T,r)$ be a locally compact $\R$-tree which is spanned by its ends at infinity and $B:=(B_t)_{t\ge 0}$ the
$\lambda^{(T,r)}$-Brownian motion on $(T,r)$. We introduce $\tau_0:=\inf\{t\ge 0:\,B_t\in V\}$, and put
$Y_0:=B_{\tau_0}$. Define then recursively for all $n\in\N$,\label{L:10}

\begin{equation}\label{e:f4}
   \tau_n
 :=
   \inf\big\{t>\tau_{n-1}\,|\,B_t\in V\setminus\{X_{n-1}\}\big\}.
\ee
and put
\begin{equation}\label{e:f4Y}
   Y_n
 :=
   B_{\tau_n}.
\ee

Then the stochastic process $Y=(Y_n)_{n\in\mathbb{N}_0}$ is a weighted Markov chain on the
weighted, discrete tree $(V,\{r_{\{x,y\}};\,x,y\in V\})$ .
\end{lemma}\sm

One can perhaps use the ``Trace Theorem'', Theorem 6.2.1 in \cite{FukushimaOshimaTakeda1994},  to prove the above lemma but we present a direct proof instead. As a preparation, we state the following lemma.
\begin{lemma}[Hitting times]
Fix a locally compact $\R$-tree $(T,r)$ {spanned by its ends at infinity} and a Radon measure $\nu$ on $(T,{\mathcal B}(T))$.
Let $B=((B_t)_{t\ge 0},({\mathbf P}^x)_{x\in T})$ be the continuous $\nu$-symmetric strong Markov process \label{L:hitting}
 on $(T,r)$
whose Dirichlet form is $({\mathcal E}, {\mathcal D}({\mathcal E}))$. Consider a branch point $x\in T$ and the finite family $\{x_1,...,x_n\}$ in $T$, for some $n\in\N$, of all branch points adjacent to $x$, i.e., $r(x_i,x_j)=r(x_i,x)+r(x,x_j)$, for all $1\le i<j\le n$ and for all $i=1,...,n$ the open arc $]x_i,x[$ does not contain further branch points.
Then the following holds:
\begin{itemize}
\item[(i)] $\mathbf{P}^x\{\wedge_{i=1}^n\tau_{x_i}<\infty\}=1$.
\item[(ii)]
For all $1\le i<j\le n$ and all $x$ in the subtree spanned by $\{x_1,...,x_n\}$,
\be{Xhit1}
   \mathbf P^x\big\{\tau=\tau_{x_i}\big\}
 =
   \frac{(r(x_i,x))^{-1}}{\sum_{j=1}^n(r(x_j,x))^{-1}}.
\ee
\end{itemize}
\end{lemma}\sm

\begin{proof}[Proof of Lemma~\ref{L:hitting}]
Let $(T,r)$, $\nu$, $n\in\mathbb{N}$, and $x_1,...,x_n$ be as by assumption.

(i) Let $D$ be the compact sub-tree formed by $x$ along with $x_{1},..., x_{n}$, and $\tau_D$ denote the exit time of $B$ from $D$, i.e. $\tau_D:=\wedge_{i=1}^{n}\tau_{x_{i}} $.  Reasoning as in the proof of Proposition \ref{P:prop}, it follows that
\begin{equation}
  \mathbf{E}^{x}\big[\int_{0}^{\tau_{D}}\mathrm{d}s\, f(B_{s})\big] = \int_{D}\nu(\mathrm{d}y)\, g^{\ast}_{D}(x,y) f(y)
   \ee
whenever $f \in L^{1}(\nu)$ and $g^{\ast}_{D}(x,\cdot)$ is the Green kernel as defined in  Definition \ref{Def:03}.  As $D$ is a non-empty compact subset of $T$, $g^{\ast}_{D} (x,\boldsymbol{\cdot})$ is a bounded function on $D$.  The result follows if we choose specifically {$f:=\mathbf{1}$.}\sm

(ii)
By Lemma~\ref{L:061}, we can choose for all $i=1,...,n$ a finite set $V_i\subset T$ such that  for all $v\in V_i$, $x_i\in[v,x]$ and $]v,x[$
does not contain any branch points.
Define then for all $i=1,...,n$ a function $h_{i}:\,T\to[0,1]$
by the following requirements: $h_i(x_i)=1$, $h_i(x):=\frac{(r(x_i,x))^{-1}}{\sum_{j=1}^n(r(x_j,x))^{-1}}$,
$h_i$ is supported on the subtree spanned by $V_i\cup \{x_1,...,x_n\}\setminus\{x_i\}$, and is linear on the arcs $[x,x_j]$, for all $j=1,...,n$, and $[v,x_i]$ for all $v\in V_i$.
%
Obviously, $h_i\in {\mathcal L}_{V_i\cup\{x_1,...,x_n\}\setminus\{x_i\},\{x_i\}}$.
Moreover, if we choose $x$ as the root,
\begin{equation}
\label{e:005}
\begin{aligned}
   &\nabla h_i
  \\
 &:=\sum_{v\in V_i}\tfrac{h_i(v)-h_i(x_i)}{r(x_i,v)}\mathbf{1}_{[v,x_i]}+\tfrac{h_i(x_i)-h_i(x)}{r(x_i,x)}\mathbf{1}_{[x,x_i]}+\sum_{j=1,j\not =i}^n\tfrac{h_i(x_j)-h_i(x)}{r(x_j,x)}\mathbf{1}_{[x,x_j]}
  \\
 &=
   -\sum_{v\in V_i}r^{-1}(x_i,v)\mathbf{1}_{[v,x_i]}+\frac{\sum_{j=1,j\not =i}^n(r(x_j,x))^{-1}}{r(x,x_i)\cdot \sum_{j=1}^n(r(x_j,x))^{-1}}\mathbf{1}_{[x,x_i]}
  \\
 &\;-\sum_{j=1,j\not =i}^n\tfrac{(r(x_i,x))^{-1}}{r(x_j,x)\cdot \sum_{k=1}^n(r(x_k,x))^{-1}}\mathbf{1}_{[x,x_j]}.
\end{aligned}
\end{equation}

Hence, for
all $g\in{\mathcal D}_{V_i\cup\{x_1,...,x_n\}}({\mathcal E})$,
\begin{equation}
\label{e:hn}
\begin{aligned}
   {\mathcal E}\big(h_i,g\big)
 &=
   \tfrac{1}{2}\sum_{v\in V_i}r^{-1}(x_i,v)\big(g(x_i)-g(v)\big)\\
&\;\ +\tfrac{1}{2}\frac{\sum_{j=1,j\not =i}^n(r(x_j,x))^{-1}}{r(x,x_i)\cdot \sum_{j=1}^n(r(x_j,x))^{-1}}\big(g(x_i)-g(x)\big)
  \\
 &\;\;+\tfrac{1}{2}\sum_{j=1,j\not =i}^n\tfrac{(r(x_i,x))^{-1}}{r(x_j,x)\cdot \sum_{k=1}^n(r(x_k,x))^{-1}}
   \big(g(x)-g(x_j)\big)
  \\
 &=
  0.
\end{aligned}
\end{equation}

By part(i) of Proposition~\ref{L:05}, this identifies $h_i$ as the unique minimizer of (\ref{cap.1}) with
$\alpha=0$, $A:=V_i\cup\{x_1,...,x_n\}\setminus\{x_i\}$ and $B:=\{x_i\}$. Hence we can conclude similarly as in the proof of Proposition~\ref{P:prop} that for all $i\in\{1,...,n\}$, the process $Y^i_t:=h_i(B_t)$ is a bounded martingale. Thus by the optional sampling theorem applied with
$\tau:=\tau_1\wedge ...\wedge\tau_{x_n}<\infty$, $\mathbf{P}^x$-almost surely. Thus
\begin{equation}
\label{equation}
   \frac{(r(x_i,x))^{-1}}{\sum_{j=1}^n(r(x_j,x))^{-1}}=\mathbf{E}^x\big[Y^{i}_0\big]=\mathbf{E}^x\big[Y^{i}_\tau\big]
 =
   \mathbf{P}^x\big\{\tau=\tau_{x_i}\big\},
\end{equation}
for all $i=1,...,n$ and the claim follows.
\end{proof}\sm

\begin{proof}[Proof of Lemma~\ref{L:10}] Without loss of generality,  we may assume that $B_0=x$ is a branch point.
Fix a vertex $x\in V$ and let $x_1,...,x_k\in V$ be the collection of all vertices
incident to $x$. It suffices to prove for $\tau:=\tau_{x_1}\wedge ...\wedge\tau_{x_k}$
and all $i\leq k$,
\begin{equation}\label{e:09.1}
   \mathbf P^x\{\tau_{x_i}=\tau\}=\big(r_{\{x,x_i\}}\pi(x)\big)^{-1},
\ee
where $\pi(x):=\sum_{x'\sim x}\tfrac{1}{r(x,x')}$,
which is the claim of Lemma~\ref{L:hitting}.
\end{proof}\sm

We conclude this section by giving the proof of Theorem~\ref{nashwillconverse}.

\begin{proof}[Proof of Theorem \ref{nashwillconverse}] By Remark~\ref{Rem:02} we can
construct a locally compact $\R$-tree which is spanned by its leaves at infinity and
with branch points in $V$ such that
on $V$ its metric coincides with $r$.
The
assertion now follows from the previous lemma in combination with
Theorem~\ref{T:trareha}.
\end{proof}\smallskip

\section{Examples and diffusions with more general scale function}
\label{s:BMdrift}

As suggested by Proposition~\ref{P:prop}, $\nu$-Brownian motion can be thought of as a diffusion on natural scale with speed measure $\nu$. We begin by listing a couple of examples, which can be found in the literature:

\begin{example}[Time changed Brownian motion on (subsets of) $\R$]\rm
Let $-\infty\le a<b\le\infty$, and let the $\R$-trees $(T,r)$ be $(a,b)$, $[a,b)$, $(a,b]$ or $[a,b]$ equipped with the Euclidian distance.
Consider the solution of the stochastic differential equation
\begin{equation}
\label{e:BM}
   \mathrm{d}X_t=\sqrt{a(X_t)}\mathrm{d}B_t,
\end{equation}
where $B:=(B_t)_{t\ge 0}$ is standard Brownian motion on the real line and $a:T\to\R_+$ a measurable function such that
\begin{equation}
\label{e:speed}
   \nu(\mathrm{d}x):=\tfrac{1}{a(x)}\mathrm{d}x
\end{equation}
defines a Radon-measure on $(T,{\mathcal B}(T))$. It is well-known that under (\ref{e:speed}), the equation (\ref{e:BM}) has a unique weak solution $X:=(X_t)_{t\ge 0}$
whose Dirichlet form is given by (\ref{con.2p}) with domain ${\mathcal D}({\mathcal E}):=L^2(\nu)\cap{\mathcal A}_\R$ where ${\mathcal A}_\R$ is the space of absolutely continuous functions that vanish at infinity.
\hfill$\qed$
\end{example}\sm

A less standard example is the Brownian motion on the CRT.
\begin{example}[$\nu$-Brownian motion on the CRT]\sm
Let $(T,r)$ be the CRT coded as an $\R$-tree. That is, let $B^{\mathrm{exc}}$ denote a standard Brownian excursion on $[0,1]$. Define an equivalence relation $\sim$ on $[0,1]$
be letting \label{Exp:02}
\begin{equation}
\label{e:CRT1}
   u\sim v\hspace{1cm}\mbox{ iff }\hspace{1cm}B^{\mathrm{exc}}_{u}=B^{\mathrm{exc}}_v=\inf_{u'\in[u\wedge v,u\vee v]}B^{\mathrm{exc}}_{u'}.
\end{equation}
Consider the following pseudo-metric on the quotient space $T:=[0,1]\big|_\sim$:
\begin{equation}
\label{e:CRT2}
   r(u,v):=2\cdot B^{\mathrm{exc}}_{u}+2\cdot B^{\mathrm{exc}}_v-4\cdot B^{\mathrm{exc}}_v=\inf_{u'\in[u\wedge v,u\vee v]}B^{\mathrm{exc}}_{u'}.
\end{equation}
By Lemma~3.1 in \cite{EvaPitWin2006} the CRT is compact, almost surely,  and thus $\nu$-Brownian motion exists if $\nu$ is a finite measure on $(T,{\mathcal B}(T))$
with $\mathrm{supp}(\nu)=T$. The following two choices for $\nu$ can be found in the literature.
\begin{itemize}
\item In \cite{Kre95} first an enumerated countable dense subset $\{e_1,e_2,...\}$ of the set $T\setminus T^o$ of boundary points is fixed, and then $\nu$ is chosen to be $\nu:=\sum_{i=1}^\infty 2^{-i}\lambda^{[\rho,e_i]}$.
\item In \cite{Cro07,Cro08} $\nu$ is chosen to be the uniform distribution on $(T,r)$ defined as the push forward of the Lebesgue measure on $[0,1]$ under the map which sends $u\in [0,1]\big|_\sim$ into the CRT as coded above.
\end{itemize}\hfill$\qed$
\end{example}\sm

In this section we consider diffusions that are not on natural scale. That is, we look for conditions on a measure $\mu$ on $(T,{\mathcal B}(T))$ such that
the form
\begin{equation}
\label{con.2pmu}
\begin{aligned}
   {\mathcal E}(f,g)
 &:=
   \frac{1}{2}\int\mu(\mathrm{d}z)\nabla f(z)
   \nabla g(z)
\end{aligned}
\ee
for all $f,g\in{\mathcal D}({\mathcal E})$ with the same domain ${\mathcal D}({\mathcal E})$ as before (compare (\ref{domainp}))
defines again a regular Dirichlet form. If this is the case we would like to refer to the corresponding diffusion as $(\mu,\nu)$-Brownian motion.

\begin{example}[Diffusion on $\R$]\rm Let $X:=(X_t)_{t\ge 0}$ be the diffusion on $\R$ with differentiable scale function $s:\R\to\R_{+}$ and speed measure $\nu:{\mathcal B}(\R)\to\R_+$. Then $X$ is the
continuous strong Markov process associated with the Dirichlet form
\begin{equation}
\label{e:007}
     {\mathcal E}(f,g)
 :=
     \tfrac{1}{2}\int\tfrac{\mathrm{d}z}{s'(z)}\cdot f'(z)\cdot g'(z)
\end{equation}
for all $f,g\in L^2(\nu)\cap{\mathcal A}{_\R}$ such that ${\mathcal E}(f,g)<\infty$ with ${\mathcal A}{_\R}$ denoting the set of all absolutely continuous functions which vanish at infinity.

It is well-known for regular diffusions that one can do a ``scale change'' resulting in a diffusion on the natural scale.
For that purpose, let for all $x,y\in\R$,
\begin{equation}
\label{e:rc}
   r_{s}(x,y)
 :=
   \int_{[x\wedge y,x\vee y]}\mathrm{d}z\,s'(z).
\ee
It is easy to see that $(\R,r_s)$ is isometric to a connnected subset of $\R$ {and therefore a locally compact} $\R$-tree which
has length measure $\mathrm{d}\lambda^{(\R,r_s)}=s'(x)\,\mathrm{d}x$.
We find that
\begin{equation}
\label{e:006}
\begin{aligned}
    {\mathcal E}(f,g)
 &=
    \frac{1}{2}\int\tfrac{\mathrm{d}z}{s'(z)}\,(s'(z)\nabla_{r_c} f(z))\cdot (s'(z)\nabla_{r_c} g(z)),
   \\
 &=
    \frac{1}{2}\int\mathrm{d}\lambda^{(\R,r_c)}\,\nabla_{r_c} f\cdot \nabla_{r_c} g,
\end{aligned}
\end{equation}
where $f,g\in L^2(\nu)\cap{\mathcal A}{_\R}$ such that ${\mathcal E}(f,g)<\infty$.
This implies that the $\nu$-Brownian motion, $B^{s}$, on $(\R,r_{s})$ has the same distribution as $X$ on $(\R,\mid\boldsymbol{\cdot}\mid)$.
Moreover, Theorems~\ref{T:04} and~\ref{T:trareha} imply that $X$ is recurrent iff $\int_0^\infty\mathrm{d}y\,s(y)=\infty$ and  $\int_{-\infty}^0\mathrm{d}y\,s(y)=\infty$.
\label{Exp:03}

Specifically, if $X^c_{t} = B_{t} + c\cdot t$ is the (standard) Brownian motion on $\R$ with drift $c\in\R$, then its
scale function is $s(x):=\int^x_{\cdot}e^{-2cy}\mathrm{d}y$ and its speed measure is $\nu(\mathrm{d}x):=e^{2cx}\mathrm{d}x$. Thus with the choice
\begin{equation}
   r_{c}(x,y)
 :=
   \tfrac{1}{2c}e^{-2c x\wedge y}(1-e^{-2c|y-x|)}),
\ee
for all $x,y\in\R$,
$X^c$ on $(\R,\mid\boldsymbol{\cdot}\mid)$ has the same distribution as $e^{2cx}\mathrm{d}x$-Brownian motion on $(\R,r_c)$.
Since
$(\R,r_c)$ is isometric
to $(\tfrac{1}{2c},\infty)$ if $c<0$ and $(-\infty,\tfrac{1}{2c})$ if $c>0$, $X^c$ is recurrent iff $c=0$.
\hfill$\qed$
\end{example}\sm

We want to formalize the notion of a ``scale change'' discussed in Example~\ref{Exp:03} on general separable $\R$-trees $(T,r)$, and consider a method by which
we could construct diffusions on $(T,r)$ which are not necessarily on natural scale.

Assume we are given a separable $\R$-tree $(T,r)$, a Radon measure $\nu$ on $(T,{\mathcal B}(T))$ and a further measure $\mu$ on  $(T,{\mathcal B}(T))$ which is absolutely continuous {with density $e^{-2\phi}$} with respect to the length measure $\lambda^{(T,r)}$. Define the form $({\mathcal E},{\mathcal D({\mathcal E})})$ with ${\mathcal E}$ as in (\ref{con.2pmu}) and ${\mathcal D}({\mathcal E})$ as in (\ref{domainp}).
In the following we will refer to a {\em potential} as a function $\phi:T\to\R$ such that {for all $a,b\in T$,}
\begin{equation}\label{e:rphi}
   r_\phi(a,b)
 :=
   \int_{[a,b]}\lambda^{(T,r)}(\mathrm{d}x)\,\mathrm e^{-2\phi(x)}<\infty,
\ee
for all $a,b\in T$. An implicit assumption in the definition being that the function $\phi$ has enough regularity for the integral above to make sense.

It is easy to check that $r_\phi$ is a metric on $T$
which generates the same topology
as $r$ and that the metric space $(T,r_\phi)$ is
also an $\R$-tree.
If the potential $\phi$ is such that
the $\R$-tree $(T,r_\phi)$ is locally compact,
then $({\mathcal E},{\mathcal D}({\mathcal E}))$ is a regular Dirichlet form, and the corresponding
$(\mu,\nu)$-Brownian motion on $(T,r)$ agrees in law with
$\nu$-Brownian motion on $(T,r_\phi)$. \sm

We close this section with the example of a diffusion which is extensively studied in \cite{Eva00}.
\begin{example}[Evans's Brownian motion on THE $\R$-tree] \rm
In \cite{Eva00} Evans constructs a continuous path Markov process on the ``richest'' $\R$-tree, which branches ``everywhere'' in ``all possible'' directions. More formally, consider the set $T$ of all bounded subsets of $\R$ that contain their supremum. Denote for all $A,B\in T$ by
\label{Exp:06}
\begin{equation}
\begin{aligned}
   &\tau(A,B)
  \\
 &:=
   \sup\big\{t\le \sup(A)\wedge\sup(B):\,(A\cap(-\infty,t])\cup\{t\}=(B\cap(-\infty,t])\cup\{t\}\big\}
\end{aligned}
\end{equation}
the ``generation'' at which the lineages of $A$ and $B$ diverge, and put
\begin{equation}
   r(A,B)
 :=
   \sup(A)+\sup(B)-2\cdot\tau(A,B).
\end{equation}
Then $(T,r)$ is a $\R$-tree which is spanned by its ends at ``infinity''. Note that $(T,r)$ is not locally compact.

Suppose that $\mu$ is a $\sigma$-finite Borel measure on $E_\infty$ such that $0<\mu(B)<\infty$ for every ball $B$ in the metric $\bar{r}(\xi,\eta):=2^{-\sup(\xi\wedge\eta)}$. In particular, the support of $\mu$ is all of $E_\infty$.
Distinguish an element $\rho\in E_\infty$. The ``root'' $\rho$ defines a partial order on $(T,r)$ in a canonical way by saying that $x\le y$ if $x\in[\rho,y]$.
For each $x\in T$, denote by $S^x:=\{\xi\in E_\infty:\,x\in[\rho,\xi]\}$, and consider the measure $\nu(\mathrm{d}x):=\mu(S^x)\lambda^{T,r}(\mathrm{d}x)$. It was shown in Section~5 in \cite{Eva00} that the measure $\nu$ is Radon. Moreover, a continuous path Markov process was constructed which is a $(\nu,\nu)$-Brownian motion on $(T,r)$ in our notion.
Hence if $\int^b_a\mathrm{d}\lambda^{(T,r)}(\mu(S^x))^{-1}<\infty$, for all $a,b\in T$, and if $(T,r_{\mathrm{natural}})$ is locally compact, where
\begin{equation}
\label{e:008}
   r_{\mathrm{natural}}(x,y):=\int_{[x\wedge y,x\vee y]}\tfrac{\lambda^{(T,r)}(\mathrm{d}z)}{\mu(S^z)},\hspace{1cm}z\in T,
\ee
then its law is the same as that of {$\nu$}-Brownian motion on $(T,r_{\mathrm{natural}})$.
\hfill$\qed$
\end{example}\sm

\bibliographystyle{alpha}

\end{document}